\documentclass[11pt,oneside]{amsart}
\usepackage{amsfonts, amssymb}
\usepackage{amsmath,amsthm,amscd,xypic}
\usepackage{color}
\usepackage{epic,eepic}
\usepackage{geometry}
\newtheorem{theorem}{Theorem}[section]
\newtheorem{thm}[theorem]{Theorem}
\newtheorem*{thm*}{Theorem}
\newtheorem{lem}[theorem]{Lemma}
\newtheorem{cor}[theorem]{Corollary}
\newtheorem*{cor*}{Corollary}

\newtheorem*{thmA}{Theorem A}
\newtheorem*{thmB}{Theorem B}
\newtheorem*{thmC}{Theorem C}
\newtheorem*{thmD}{Theorem D}

\theoremstyle{definition}

\newtheorem{example}[theorem]{Example}

\newtheorem{conj}[theorem]{Conjecture}

\theoremstyle{remark}
\newtheorem{remark}[theorem]{Remark}
\numberwithin{equation}{section}

\newcommand{\vol}[1]{\text{vol}\left(#1\right)}

\newcommand{\snk}{\text{sn}_{\kappa}}

\newcommand{\Alex}{\text{Alex\,}}

\newcommand{\Alexnk}{\text{Alex}^n(\kappa)}
\newcommand{\dsp}{\displaystyle}
\newcommand{\geod}[1]{[\,#1\,]}
\newcommand{\drn}{\uparrow}

\newcommand{\scup}[2]{\underset{#1}{\overset{#2}\cup}}

\begin{document}

\title{Relatively Maximum Volume Rigidity in Alexandrov Geometry}

\author{Nan Li}
\address{Department of Mathematics, University of Notre Dame, Notre Dame, IN 46556}
\email{nli2@nd.edu}
\thanks{Both authors are supported
partially by NSF Grant DMS 0805928 and by a research found from
Capital Normal University.}

\author{Xiaochun Rong}
\address{Department of Mathematics, Capital Normal University, Beijing, P.R. China}
\address{Department of Mathematics, Rutgers University, New Brunswick, New Jersey 08854}
\email{rong@math.rutgers.edu}





\begin{abstract}
Given a compact metric space $Z$ with Hausdorff dimension $n$, let $X$ be a metric space such that there is a distance non-increasing onto map $f:Z\to X$. Then the Hausdorff $n$-volume $\vol X\le\vol Z$. The {\it relatively maximum volume conjecture} says that if $X$ and $Z$ are both Alexandrov spaces and $\vol X=\vol Z$, then $X$ is isometric to a gluing space produced from $Z$ along its boundary $\partial Z$ and $f$ is length preserving. We will partially verify this conjecture, and give a further classification for compact Alexandrov $n$-spaces with relatively maximum volume in terms of a fixed radius and space of directions. We will also give an elementary proof for a pointed version of Bishop-Gromov relative volume comparison with rigidity in Alexandrov geometry.
\end{abstract}

\maketitle

\section*{Introduction}

Let $Z$ be a compact metric space with Hausdorff dimension $\alpha$. Consider all compact metric spaces
$X$ with Hausdorff dimension $\alpha$ such that there is a distance non-increasing onto map $f: Z\to X$. We let ``vol'' denote the Hausdorff measure (or volume) in the top dimension. Then $\vol X\le \vol Z$.
A natural question is to determine $X$ (in terms of $Z$)
when $\text{vol}(X)=\text{vol}(Z)$. We will refer this as a {\it relatively maximum volume rigidity problem}.

A possible answer to the relatively maximum volume rigidity problem is closely related to the regularity of
underlying geometric and topological structures. For instance, if $Z$ and $X$ are closed
Riemannian $n$-manifolds, then $f$ is an isometry (see Corollary \ref{cor1.A}).
On the other hand, taking any measure-zero subset $S$ in $Z$ (a Riemannian manifold) and identifying $S$ with a point $p\in S$, then the projection map, $Z\to X=Z/(S\sim p)$, is a distance non-increasing onto map, and it is hopeless to have some rigidity
on $Y$ in terms of $X$.

In this paper, we will study the relatively maximum volume rigidity problem in Alexandrov
geometry, partly because an Alexandrov space $X$ has a ``right" geometric structure
for this problem (see Conjecture 0.1 below). For instance,
for $p\in X$, the gradient-exponential map, $g\exp_p: T_pX\to X$,
becomes a distance non-increasing map, when $T_pX$ is equipped with
the $\kappa$-cone metric via the cosine law on the space form $S^2_\kappa$ (cf. [BGP]).
When taking $Z$ to be a closed $r$-ball at the vertex (for $\kappa>0$, $r\le \frac
\pi{2\sqrt \kappa}$ or $r=\frac \pi{\sqrt \kappa}$), the relatively volume
rigidity problem (see Theorem B) indeed extends the (absolutely) Maximum
Radius-Volume Rigidity Theorem proved by
Grove-Petersen ([GP], Theorem \ref{[GP]}).

The recent study of Alexandrov spaces was initiated by Burago-Gromov-Perel'man
in the paper [BGP] and has gotten a lot attention lately. An Alaxandrov space with curvature $\text{curv}\ge
\kappa$ is a length metric space such that each point has a neighborhood
in which Toponogov triangle comparison holds with respect to the space form of constant curvature
$\kappa$. In the rest of the paper, we will freely use basic notions on an Alexandrov space
from [BGP] and [Pet2] (e.g., the space of directions, the gradient-exponential maps,
$(n,\delta)$-strained
points, etc). Let $\text{Alex}^n(\kappa)$ denote the collection of compact
Alexandrov $n$-spaces with $\text{curv}\ge \kappa$.


Note that the boundary gluing will automatically yield a distance non-increasing onto (projection) map, which also preserves the volume (see Example \ref{eg1}, \ref{eg2}). We propose the following relatively maximum volume rigidity conjecture for Alexandrov spaces.


\begin{conj} Let $Z$, $X\in \text{Alex}^n(\kappa)$, and let
$f: Z\to X$ be a distance non-increasing onto map. If $\text{vol}(Z)=\text{vol}(X)$,
then $X$ is isometric to a gluing space produced from $Z$ along its boundary $\partial Z$ and $f$ is length preserving. In particular, $Z$ is isometric to $X$
if $\partial Z=\varnothing$ or if $f$ is injective.

\end{conj}




Our goal in this paper is to partially verify Conjecture 0.1, and give a classification for the boundary gluing maps in a special case (see Theorem A, Corollary \ref{cor1.A}
and Theorem B).

We now begin to state the main results in this paper. Throughout this paper, $\tau(\delta)$ denotes a function in $\delta$ such that $\tau(\delta)\to 0$ as $\delta\to 0$. Our first
result verifies conjecture 0.1 for the case that
$f$ preserves non-$(n,\delta)$-strained points up to an error $\tau(\delta)$. For $X\in \text{Alex}^n(\kappa)$ and
$\delta>0$, let $X^\delta\subseteq X$ denote the set of all
$(n,\delta)$-strained points. Then a small ball centered at an
$(n,\delta)$-strained point is almost isometric to an open subset
in $\Bbb R^n$ ([BGP]).


\begin{thmA} Let $Z$, $X$ be Alexandrov $n$-spaces (not necessarily complete) with
curvature $\text{curv}\ge \kappa$ and $\text{vol}(Z)=\text{vol}(X)$.
Suppose that $f: Z\to X$ is a distance non-increasing onto map
such that for any $\delta>0$, $f^{-1}(X^\delta)\subseteq Z^{\tau(\delta)}$. Then $f$ is an isometry.
\end{thmA}



A point $z$ in $Z$ is called {\it regular}, if the space of directions
$\Sigma_x$ is isometric to a unit sphere. Clearly, the space $Z$ with all
points regular is a topological manifold but $Z$ may not be
isometric to any Riemannian manifold (e.g., the doubling of
two flat disks). Theorem A includes the following case:


\begin{cor} \label{cor1.A} Let $Z, X\in \text{Alex}^n(\kappa)$ with $\text{vol}
(Z)=\text{vol}(X)$ and all points in $Z$ are regular (e.g $Z$ is a Riemannian manifold). If $f: Z\to X$
is a distance non-increasing onto map, then $f$ is an isometry.
\end{cor}


In Alexandrov geometry, perhaps the most natural distance non-increasing onto map is the
gradient-exponential map $g\exp_p: C_\kappa(\Sigma_p)\to X$, $p\in X\in \text{Alex}^n(\kappa)$,
where $C_\kappa(\Sigma_p)$ denotes the tangent cone $T_pX$ equipped
with a $\kappa$-cone metric via the cosine law in $S^2_\kappa$ ([BGP]).
Since $g\exp_p$ is distance non-increasing and preserves any $r$-ball,
one immediately gets the pointed version of Bishop type
volume comparison:
$$\text{vol}(B_R(p))\le \text{vol}(C^R_\kappa(\Sigma_p)),$$
where $C^R_\kappa(\Sigma_p)$ denotes the {\it open} $R$-ball in
$C_\kappa(\Sigma_p)$ at the vertex $\tilde o$. We will show that when the equality holds, $g\exp_p$ will satisfy the conditions in Theorem A (Lemma \ref{exp.delta}, Lemma \ref{exp})
and thus open ball $C^R_\kappa(\Sigma_p)$ is isometric to $B_R(p)$
with respect to intrinsic metrics (see Theorem \ref{open.iso}).

As an important case for Conjecture 0.1, one leads to classify Alexandrov
spaces with relatively maximum volume: given any $\kappa$, $R>0$ and $\Sigma\in
\text{Alex}^{n-1}(1)$, let $\mathcal A^R_\kappa(\Sigma)$
denote the collection of Alexandrov $n$-spaces $X\ni p$ satisfying
$$\text{curv}\ge \kappa,\qquad X=\bar B_R(p),\qquad \Sigma_p=\Sigma.$$
Then $\text{vol}(X)\le \text{vol}(C^R_\kappa(\Sigma))=v(\Sigma,\kappa,R)$.
When $\text{vol}(X)=v(\Sigma,\kappa,R)$, we say that $X$ has the
relatively maximum volume.


\begin{thmB}[{Relatively maximum volume rigidity}] Let $X\in
\mathcal A^R_\kappa(\Sigma)$ such that $\text{vol}(X)=v(\Sigma,\kappa,R)$.
Then $X$ is isometric to $\bar C^R_\kappa(\Sigma)/x\sim \phi(x)$
and $R\le \frac \pi{2\sqrt \kappa}$ or $R=\frac \pi{\sqrt \kappa}$
for $\kappa>0$, where $\phi: \Sigma\times \{R\}\to\Sigma\times \{R\}$
is an isometric involution (which can be trivial). Conversely, given
any isometric involution $\phi$ on $\Sigma$, $\bar C^R_\kappa(\Sigma)
/x\sim \phi(x)\in \mathcal A^R_\kappa(\Sigma)$
and has the relatively maximum volume.
\end{thmB}


Theorem B verifies Conjecture 0.1 for the case $f=g\exp_p: Z=\bar C^R_\kappa(\Sigma_p)\to
X$, together with a further classification for the boundary identification.
Note that Theorem B implies that if $k>0$ and $\frac \pi{2\sqrt\kappa}<R<\frac \pi{\sqrt \kappa}$, then $\max
\{\text{vol}(X),\,\, X\in\mathcal A^R_\kappa(\Sigma)\}<v(\Sigma,\kappa,R)$. For the case that $X$ is a limit of Riemannian manifolds, a classification was given in [GP]. A general classification is more complicated, and we wish to discuss it elsewhere.

As mentioned earlier, Theorem B extends the radius-volume rigidity theorem in [GP],  which are stated below.


\begin{thm}[{[GP]}]\label{[GP]} Let $M_i\overset{d_{GH}}\longrightarrow X$ be a Gromov-Hausdorff
convergent sequence of Riemannian $n$-manifold such that
$$\text{sec}_{M_i}\ge \kappa,\qquad \text{rad}(M_i)=R,\qquad \text{vol}
(M_i)\to \text{vol}
(C^R_\kappa(S^{n-1}_1)),$$
where $\text{rad}(M_i)=\min\{r, \bar B_r(p)=M_i, \,\, p\in M_i\}$. Then $X$ is isometric to $\bar C^R_\kappa(S^{n-1}_1)/x\sim \phi(x)$ and
$R\le \frac \pi{2\sqrt \kappa}$ or $R=\frac \pi{\sqrt \kappa}$ for $\kappa>0$,
where
$\phi: \partial \bar C^R_\kappa(S_1^{n-1})\to \partial \bar C^R_\kappa(S^{n-1}_1)$ is
either the antipital map, or a reflection by a totally geodesic
hypersurface. Moreover, $M_i$ is homeomorphic to an
$n$-sphere or a real projective $n$-space.
\end{thm}


Note that $\text{vol}(X)=\text{vol}(C^R_\kappa(S^{n-1}_1))$. Choosing $p_i\in M_i$
such that $M_i=\bar B_R(p_i)$, then $p_i\to p\in X$ and $\Sigma_p=S^{n-1}_1$. By now
Theorem B implies the rigidity part of Theorem \ref{[GP]} (a generalization of the homeomorphic rigidity
in Theorem \ref{[GP]} will be given in Theorem C). Theorem B
also implies the following extension of Theorem \ref{[GP]} by S. Shteingold.

\begin{thm} [{[Sh]}] Let $X\in \mathcal A^r_\kappa(S^{n-1}_1)$ with
$\text{vol}(X)=v(S^{n-1}_1,\kappa,r)$. Then $X=\bar C^r_\kappa(S^{n-1}_1)/x\sim\phi(x)$,
$x\in S^{n-1}_1\times \{r\}$,
where $\phi$ is the reflection on a $\ell$-dimensional
totally geodesic subsphere, $1\le \ell\le n$ ($\phi$ is
trivial for $\ell=n$.)
\end{thm}


A further problem concerning Theorem B is to determine the homeomorphic type of $X$.
We have solved this problem for $X$ being a topological manifold (see Theorem \ref{[GP]}).


\begin{thmC} Given $\Sigma\in \text{Alex}^{n-1}(1)$, $\kappa$ and $R>0$, there
is a constant $\epsilon=\epsilon(\Sigma,\kappa,R)>0$ such that if
$X\in \mathcal A^R_\kappa(\Sigma)$ with $\text{vol}(X)>
v(\Sigma,\kappa,R)-\epsilon$ and $X$ is a closed topological manifold, then $X$ is
homeomorphic to $S^n_1$ or a real projective space
$\Bbb RP^n$.
\end{thmC}


Note that $\Sigma$ in Theorem C is not necessarily a topological manifold;
for instance, $X=C_1(C_1(N))$, the twice
spherical suspensions over a Poincar\'e sphere $N$, satisfies
Theorem C but $\Sigma=C_1(N)$ is not a topological manifold.
However, $X$ is homeomorphic to a $5$-sphere (cf. [Ka1]).

In the proof of Theorem B, we establish a pointed version of Bishop volume comparison with rigidity (Theorem \ref{open.iso}). In general, we will prove the following pointed version of Bishop-Gromov relative volume comparison with rigidity.

For $p\in X\in\Alexnk$, let $A_R^r(p)$ denote the annulus $\{x\in X: r<|px|<R\}$,
$0\leq r<R$, and let $A_R^r(\Sigma_p)$ denote the
corresponding annulus in $C_{\kappa}(\Sigma_p)$.


\begin{thmD} [{Pointed Bishop-Gromov
relative volume comparison}] Let $X\in\text{Alex}^n(\kappa)$. Then for any $p\in X$
and $R_3>R_2>R_1\geq 0$,
$$\frac{\text{vol\,}(A_{R_3}^{R_1}(p))}
{\text{vol\,}(A_{R_3}^{R_2}(p))}\ge
\frac{\text{vol\,}(A_{R_3}^{R_1}(\Sigma_p))}
{\text{vol\,}(A_{R_3}^{R_2}(\Sigma_p))}\hskip1mm,\text{ or equivalently },\hskip1mm
\frac{\text{vol\,}(A_{R_2}^{R_1}(p))}
{\text{vol\,}(A_{R_3}^{R_2}(p))}\ge
\frac{\text{vol\,}(A_{R_2}^{R_1}(\Sigma_p))}
{\text{vol\,}(A_{R_3}^{R_2}(\Sigma_p))}.$$
In particular,
$$\frac{\text{vol\,}(B_{R_1}(p))}{\text{vol\,}(B_{R_3}(p))}\ge
\frac{\text{vol\,}(C_{\kappa}^{R_1}(\Sigma_p))} {\text{vol\,}
(C_\kappa^{R_3}(\Sigma_p))}.$$
If any of the above inequalities becomes equal, then the open ball
$B_{R_3}(p)$ is isometric to $C_{\kappa}^{R_3}(\Sigma_p)$ with respect to
intrinsic metrics.
\end{thmD}


\begin{remark} The Riemannian version of Bishop-Gromov relative
comparison for Alexandrov spaces (i.e., the model space is $S^n_\kappa$)
was stated in [BGP] (cf. [BBI]). A notable
difference between Theorem D and the Riemannian version
is in the rigidity part: the later is the {\it absolute maximum}
volume rigidity and its model space is {\it unique}, while the former may
be viewed as the {\it relatively maximum} volume rigidity (relatively to
$\Sigma_p$), whose model spaces are of {\it infinitely many} possibilities. Moreover, the proof of Theorem D is considerably difficult; for instance, a dimension-inductive argument (which works in the Riemannian version) does not work.
\end{remark}


\begin{remark} By Lemma \ref{open.iso} in [LR], we see that $\frac{\text{vol}(C^R_\kappa(\Sigma_p))}{\text{vol}
(C^r_\kappa(\Sigma_p))}=\frac{\text{vol}(B_R(S^n_\kappa))}{\text{vol}
(B_r(S^n_\kappa))}$ and thus the monotonicity part of Theorem D coincides with that in the Riemannian version.
We point out that our proof of the volume ratio monotonicity in
Theorem D is different from one suggested by [BGP]; we take
an elementary (calculus) approach via finding a (unconventional)
partition suitable for triangle comparison arguments while a proof
in [BBI] relies on a co-area formula for Alexandrov spaces.
\end{remark}


We now give some indication on our approach to Theorem A and Theorem B. In the proof of Theorem A,
we shall show that $f$ is a homeomorphism and $f$ preserves the length of curves. Based on
basic properties of an Alexandrov space (not necessarily complete), any curve $c$ in $X$ can be approximated
by piecewise geodesics $c_i$ in $X^{\delta_i}$
($\delta_i\to 0$) such that lengths $L(c_i)\to L(c)$. Thus, it suffices to show
that when restricting to $f^{-1}(X^\delta)$ and $X^\delta$ respectively, $f$ is injective
and $f^{-1}$ preserves the length of any geodesic up to an error $\tau(\delta)\to 0$ as $\delta\to 0$
respectively.
We derive this with a volume formula for a tube-like $\epsilon$-balls in
$X^\delta$ which can be treated as a replacement of the volume formula of a thin tube
around a curve. The proof of the volume formula
is a based on the fact that a small ball at an $(n,\delta)$-strained
point can be almost isometrically embedded into $\Bbb R^n$
(see [BGP]).

Our approach to Theorem B consists of two steps: first, establishing
the open ball rigidity: the gradient-exponential map $g\exp_p: C^R_\kappa
(\Sigma_p)\to B_R(p)\subset X$ is an isometry with respect to the intrinsic
distance. We achieve this by showing that $g\exp_p$ satisfies the
condition in Theorem A (see Lemma \ref{exp.delta} and Lemma \ref{exp}). Consequently,
$X=\bar C^R_\kappa (\Sigma_p)/\sim$, where $\sim$ is a relation
on $\Sigma_p\times \{R\}$: $\tilde x\sim \tilde y$ if and only if $g\exp_p(\tilde x)
=g\exp_p(\tilde y)$.
Observe that if $\tilde x\ne \tilde y\in \Sigma_p\times \{R\}$ with
$\tilde x\sim \tilde y$, then the $g\exp_p$-images of the two geodesics $\geod{\tilde o
\tilde x}$ and $\geod{\tilde o\tilde y}$
together form a local geodesic at $g\exp_p\tilde x=g\exp_p\tilde y$.
Because a geodesic does not bifurcate, any equivalent class contains at most
two points and thus we obtain an involution $\phi: \Sigma_p\times
\{R\}\to \Sigma_p\times \{R\}$ such that $X=\bar
C^R_\kappa(\Sigma)/\tilde x\sim \phi(\tilde x), \tilde x\in \Sigma_p\times \{R\}$.
The main difficulty is to show that $\phi$ is an isometry. Our main technical lemma is to show that $\phi$ is almost 1-bi-Lipschitz up to a uniform error: $\left|\frac{|\phi(\tilde x)\phi(\tilde y)|}{|\tilde x\tilde y|}-1\right|\le 20|\tilde x\tilde y|$ for
$|\tilde x\tilde y|$ small (see Lemma \ref{non-cross}). This implies that $\phi$ is continuous and preserves the length of a path, and thus $\phi$ is distance non-increasing. Consequently, $\phi$ is an isometry since $\phi$ is an involution. Note that without the curvature lower bound, this in general does not imply that the metric on $X=\bar
C^R_\kappa(\Sigma)/\tilde x\sim \phi(\tilde x)$ coincides with the induced metric. For example, $X=\bar C^1_0(\mathbb S^1_1)/(\tilde x\sim \tilde x)=\bar B_1(\mathbb R^2)$ is equipped with the length metric coincides with the Euclidean metric when restricted to the interior, and $L(\gamma)$ is a half of the Euclidean arc length for any $\gamma\subset\partial X$. Our proof relies on the curvature lower bound as well as the cone metric.

Let $L_p(X)=g\exp_p(\Sigma\times\{R\})$, which locally divides a tubular neighborhood of $L_p(X)$ into two components $U_1$, $U_2$. The main difficulty in proving the above inequality is that a geodesic in $X$ connecting 2 points $a,b\in L_p(X)$ may intersect with $L_p(X)$ at many points other than $a,b$ (called {\it crossing points}). We show that if a geodesic is not contained in $L_p(X)$, then the crossing points are discrete (Corollary \ref{stay.bdy}). Thus we can reduce the proof to the case that $c_1=\geod{ab}\subset U_1$ has no crossing point. It's sufficient to construct a non-crossing piece-wise intrinsic geodesic $c_2\subset U_2$ connecting $a,b$, and show that length$(c_2)$ is close to length$(c_1)=|ab|$ up to a second order error (Lemma \ref{non-cross}).

We remark that
the present proof, in an essential way, relies on the $\kappa$-cone
metric structure; and we believe that establishing a similar inequality
in general will be the main obstacle in Conjecture 0.1.


The rest of the paper is organized as follows:

In Section 1, we will prove Theorem A.

In Section 2, we will prove Theorem B.

In Section 3, we will prove Theorem C.

In Section 4, we will prove Theorem D.

\vskip4mm

\section{$(n.\delta)$-strained isometry}

\vskip4mm

Let $f: Z\to X$ be as in Theorem A. We will establish that $f$ is an isometry
through the following properties:

\noindent (i) If a distance non-increasing onto map $f$ preserves the volume of the total spaces, then $f$ and
$f^{-1}$ preserve volumes of any subsets (see Lemma \ref{vol.preserve}).

\noindent (ii) Based on a local bi-Lipschitz embedding property
(see Lemma \ref{bgp9.4}), we show that for $\delta$ suitably
small, $f$ is injective on $f^{-1}(X^\delta)\subseteq Z^{\tau(\delta)}$. In particular, for any curve
$c\subset X^\delta$, $f^{-1}(c)\subseteq Z^{\tau(\delta)}$ is a curve (see Lemma
\ref{f.homeo}).

\noindent (iii) Our main technical lemma is a volume formula for a `tube' of
$\epsilon$-balls (which can be treated as a replacement for an $\epsilon$-tube around
a curve, see Lemma \ref{balltube}). Together with (i) and (ii), this formula
implies that $f^{-1}$ preserves the length of any geodesic in $X^{\delta}$ up
to an error $\tau(\delta)$.
Because for any small $\delta$ ($\delta<\frac 1{8n}$), $X^\delta$ is dense in $X$ (see Lemma \ref{delta.lip}), we are able to show that $f$ is also distance non-decreasing
and thus $f$ is an isometry.


\begin{lem}\label{vol.preserve} Let $f: Z\to X$ be a distance non-increasing onto map of two metric spaces
of equal Hausdorff dimension. If $\text{vol}(X)=\text{vol}(Z)$, then for any
subset $A\subseteq Z$ and $B\subseteq X$,
$$\text{vol}(A)=\text{vol}(f(A)),\quad \text{vol}(B)=\text{vol}(f^{-1}(B)).$$
\end{lem}


\begin{proof} We argue by contradiction; assuming that $\text{vol}(A)>\text{vol}(f(A))$.
Then
\begin{align*}
\text{vol}(Z)
&=\text{vol}(A)+\text{vol}(Z-A)>\text{vol}(f(A))+\text{vol}(f(Z-A))
\\
&\ge \text{vol}(f(Z))=\text{vol}(X),
\end{align*}
a contradiction. Similarly, one can check that $\text{vol}(f^{-1}(B))=\text{vol}(B)$.
\end{proof}


Let $X^\delta(\rho)$ denote the union of points with an $(n,\delta)$-strainer $\{(a_i,b_i)\}$
of radius $\rho>0$, where $\dsp\rho=\min_{1\le i\le n}\{|pa_i|, |pb_i|\}>0$.


\begin{lem}[{[BGP]} Theorem 9.4]\label{bgp9.4}
Let $X\in \Alexnk$. If $p\in X^\delta(\rho)$, then the
map $\psi:X\rightarrow \Bbb R^n$ defined by $\psi(x)
  =(|a_1x|, \cdots, |a_nx|)$ maps a small neighborhood $U$ of $p$ $\tau(\delta,\delta_1)$-almost isometrically
  onto a domain in $\Bbb R^n$, i.e. $||\psi(x)\psi(y)|-|xy||<\tau(\delta,\delta_1)|xy|$ for any
  $x,y\in U$, where $\delta_1=\rho^{-1}\cdot\text{diam}(U)$. In particular, $\psi$ is an $\tau(\delta)$-almost isometric
embedding when restricting to $B_{\delta\rho}(p)$.
\end{lem}

A consequence of Lemma \ref{bgp9.4} is that
$$1-\tau(\delta)\le \frac{\text{vol}(B_\epsilon(p))}{\text{vol}(B_\epsilon(\Bbb R^n))}
\le 1+\tau(\delta),$$
for any $p\in X^\delta(\rho)$ and $\epsilon\le\delta\rho$.

\begin{lem} \label{f.homeo} Let the assumptions be as in Theorem A. Then
$f: f^{-1}(X^\delta)\to X^\delta$ is injective. Consequently, if $\gamma\subset
X^\delta$ is a continuous curve, then $f^{-1}(\gamma)$ is also a continuous curve.
\end{lem}

\begin{proof} We argue by contradiction; assuming $z_1\ne z_2\in f^{-1}(X^\delta)$
such that $f(z_1)=f(z_2)=x$. We may assume that $z_1$ and $z_2$ have
$\tau(\delta)$-strainer of radius $\rho>0$. Choose $4\epsilon<|z_1z_2|$ and $\epsilon<\delta\rho$.
By Lemma \ref{vol.preserve} and the above consequence of Lemma \ref{bgp9.4}, we get
$$1=\frac{\text{vol}(f^{-1}(B_\epsilon(x)))}{\text{vol}(B_\epsilon(x))}\ge
\frac{\text{vol}(B_\epsilon(z_1))+\text{vol}(B_\epsilon(z_2))}{\text{vol}
(B_\epsilon(x))}\ge 2(1-\tau(\delta)),$$
a contradiction.
\end{proof}

We now develop a formula which estimates the volume of an $\epsilon$-ball tube with a higher order error. Let $x_1,x_2,\dots,x_{N+1}$ be $N+1$ points in $X^\delta(\rho)$. We first give an estimate of the volume of the $\epsilon$-ball tube $\bigcup_{i=1}^{N+1} B_\epsilon(x_i)$ in terms of $\dsp\sum_{i=1}^N|x_ix_{i+1}|$ and $\epsilon$, $\delta$ with errors.

\begin{lem}[volume of an $\epsilon$-ball tube]\label{balltube}
Let $X\in\Alexnk$ and $x_i\in X^\delta(\rho)$,
  $i=1,2,\cdots,N+1$ satisfy that $0<|x_ix_{i+1}|< 2\epsilon\ll\delta\rho$ and $B_\epsilon(x_i)\cap
B_\epsilon(x_j)\cap B_\epsilon(x_{k})=\varnothing$ for $i\neq j\neq k$.
Then the volume of the $\epsilon$-ball tube $\bigcup_{i=1}^{N+1}B_\epsilon(x_i)$
(see Figure 1) satisfies:
\begin{align}
&(1+\tau(\delta))\cdot
\vol{\bigcup_{i=1}^{N+1}B_\epsilon(x_i)}
\notag\\
    &\qquad=\vol{B_\epsilon(\mathbb R^n)} +2\epsilon\cdot\vol{B_\epsilon(\mathbb R^{n-1})}\sum_{i=1}^N\int_{\theta_i}^{\frac\pi2}\sin^n(t)dt,
\label{balltube.e01}
\end{align}
  where $\theta_i\in[0,\frac\pi2]$ such that $\cos\theta_i=\frac{|x_ix_{i+1}|}{2\epsilon}$. If in addition, $|x_ix_{i+1}|\le \epsilon^2$ for all $1\le i\le N$, then
\begin{align}
&(1+\tau(\delta))\cdot
\vol{\bigcup_{i=1}^{N+1}B_\epsilon(x_i)}
\notag\\
&\qquad=\vol{B_\epsilon(\mathbb R^n)}
+\vol{B_\epsilon(\mathbb R^{n-1})}
\sum_{i=1}^N|x_ix_{i+1}|+O(\epsilon^{n+1})\sum_{i=1}^N|x_ix_{i+1}|,
\label{balltube.e02}
\end{align}

\end{lem}
\begin{center}
\begin{picture}(300,160)
\put(130,5){Figure 1}
\put(130,140){$B_{\delta\rho}(x_i)$}
\put(240,30){$\dsp\bigcup_{i=1}^{N+1}B_\epsilon(x_i)$}
\put(115,60){$x_{i-1}$}
\put(139.5,92){$x_i$}
\put(165,60){$x_{i+1}$}
\put(15,60){$x_{1}$}
\put(265,60){$x_{N+1}$}
\multiput(20,70)(50,0){6}{
\dashline{2}(0,-20)(0,20)
\circle*{1}}
\multiput(70,70)(50,0){4}{
    \qbezier[6](	19.7538	,	-3.1287	)(	21.0497	,	5.0538	)(	16.1804	,	 11.7558	 )
    \qbezier[6](	17.5261	,	9.6351	)(	13.5351	,	16.8948	)(	5.5798	,	 19.2060	 )
    \qbezier[6](	-16.1803	,	11.7557	)(	-21.0498	,	5.0534	)(	 -19.7538	 ,	 -3.1288	 )
    \qbezier[6](	-6.1803	,	19.0211	)(	-14.0593	,	16.4611	)(	-17.8202	 ,	 9.0798	 )
\arc{40}{0.16}{2.96}}
\multiput(70,70)(50,0){4}{\arc{40}{4.44}{4.98}}
\put(20,70){
    \qbezier[6](	19.7538	,	-3.1287	)(	21.0497	,	5.0538	)(	16.1804	,	 11.7558	 )
    \qbezier[6](	17.5261	,	9.6351	)(	13.5351	,	16.8948	)(	5.5798	,	 19.2060	 )
\arc{40}{0.16}{4.98}}
\put(270,70){
    \qbezier[6](	-16.1803	,	11.7557	)(	-21.0498	,	5.0534	)(	 -19.7538	 ,	 -3.1288	 )
    \qbezier[6](	-6.1803	,	19.0211	)(	-14.0593	,	16.4611	)(	-17.8202	 ,	 9.0798	 )
\arc{40}{4.44}{2.96}}

\multiput(45,86)(50,0){5}{\circle*{1}}

\multiput(45,86)(50,0){5}{
    \qbezier[6](	-19.8005	,	2.8180	)(	-20.9678	,	-5.3838	)(	 -15.9937	 ,	 -12.0085	 )
    \qbezier[6](	-18.0965	,	-8.5156	)(	-14.5692	,	-16.0116	)(	 -6.7747	 ,	 -18.8177	)
    \qbezier[6](	16.1803	,	-11.7557	)(	21.0498	,	-5.0534	)(	19.7538	,	 3.1288	 )
    \qbezier[6](	6.1803	,	-19.0211	)(	14.0593	,	-16.4611	)(	 17.8202	 ,	 -9.0798	 )
\dashline{2}(-19,2)(-5,-19)
\dashline{2}(19,2)(5,-19)
\dashline{2}(0,-20)(0,20)
\arc{40}{3.32}{6.12}
}
\multiput(45,86)(50,0){5}{\arc{40}{1.3}{1.84}}

  \qbezier[30](	235.0000	,	86.0000	)(	235.0000	,	123.2800	)(	 208.6396	 ,	 149.6400	 )
  \qbezier[30](	81.3604	,	149.6400	)(	55.0000	,	123.2800	)(	55.0000	,	 86.0000	 )
  \qbezier[30](	55.0000	,	86.0000	)(	55.0000	,	48.7200	)(	81.3604	,	 22.3600	 )
  \qbezier[30](	208.6396	,	22.3600	)(	235.0000	,	48.7200	)(	235.0000	 ,	 86.0000	 )
\put(148,112){\small $\Gamma^+(x_i)$}
\put(114,112){\small $\Gamma^-(x_i)$}
\put(-12,96){\small $A^-(x_1)$}
\put(270,96){\small $A^+(x_{N+1})$}
\end{picture}
\end{center}

Because $B_\epsilon(x_{i-1})\cup B_\epsilon(x_i)\cup B_\epsilon(x_{i+1})\subset B_{\delta\rho}(x_i)$, which is $\tau(\delta)$-almost isometrically embedded into $\mathbb R^n$, one can divide $\scup{i=1}{N+1}B_\epsilon(x_i)$ into small pieces $\Gamma^\pm(x_i)$, whose volumes are $(1+\tau(\delta))$-proportional to the volumes of the following ``trapezoidal balls" $\Gamma_\epsilon^{h_i^\pm}(\mathbb R^n)$ in $\mathbb R^n$. This allows us to reduce the calculation to the Euclidean space.


We define the trapezoidal ball $\Gamma_r^h(\mathbb R^n)$ in $\mathbb R^n_+=\{(x_1,x_2,\cdots,x_n): x_n\ge 0\}$ as the following. Let $u\in \mathbb R^n_+$ be a point with $|ou|=h\le r$. Then the hyper plane $H$ passing through $u$ and perpendicular to $\overrightarrow{ou}$ divides the half ball $B_r(\mathbb R^n)\cap \mathbb R^n_+$ into two subsets. Let $\Gamma_r^h(\mathbb R^n)$ be the subset which contains the origin (see Figure 3). It's easy to see that $\vol{\Gamma_r^h(\mathbb R^n)}$ depends only on $h$ and $r$, but not the direction $\overrightarrow{ou}$ as long as $H\cap B_r(\mathbb R^n)\subset \mathbb R^n_+$.

\vskip -2cm
\begin{center}
\begin{picture}(100,100)
\setlength{\unitlength}{0.9pt}
\put(135,0){Figure 2}
\put(137,60){$\Gamma_r^h(\mathbb R^n)$}
\put(60,40){\arc{120}{3.1}{3.5}}
\put(60,40){\arc{120}{4.9}{6.3}}
  \qbezier[12](	120.0000	,	40.0000	)(	120.0000	,	64.8533	)(	102.4264	 ,	 82.4267	 )
\qbezier[12](	102.4264	,	82.4267	)(	84.8529	,	100.0000	)(	60.0000	,	 100.0000	 )
\qbezier[12](	60.0000	,	100.0000	)(	35.1471	,	100.0000	)(	17.5736	,	 82.4267	 )
\qbezier[12](	17.5736	,	82.4267	)(	0.0000	,	64.8533	)(	0.0000	,	 40.0000	 )

  \qbezier[10](	120.0000	,	40.0000	)(	120.0000	,	48.2844	)(	102.4264	 ,	 54.1422	 )
  \qbezier[10](	102.4264	,	54.1422	)(	84.8529	,	60.0000	)(	60.0000	,	 60.0000	 )
  \qbezier[10](	60.0000	,	60.0000	)(	35.1471	,	60.0000	)(	17.5736	,	 54.1422	 )
  \qbezier[10](	17.5736	,	54.1422	)(	0.0000	,	48.2844	)(	0.0000	,	 40.0000	 )
\qbezier(	0.0000	,	40.0000	)(	0.0000	,	31.7156	)(	17.5736	,	25.8578	 )
\qbezier(	17.5736	,	25.8578	)(	35.1471	,	20.0000	)(	60.0000	,	20.0000	 )
  \qbezier(	60.0000	,	20.0000	)(	84.8529	,	20.0000	)(	102.4264	,	 25.8578	 )
  \qbezier(	102.4264	,	25.8578	)(	120.0000	,	31.7156	)(	120.0000	,	 40.0000	 )

\qbezier[10](	71.8693	,	98.8143	)(	70.6267	,	100.9666	)(	59.6815	,	 96.6766	 )
\qbezier[10](	59.6815	,	96.6766	)(	48.7363	,	92.3866	)(	34.5000	,	 84.1673	 )
\qbezier[10](	34.5000	,	84.1673	)(	20.2637	,	75.9480	)(	11.0759	,	 68.6142	 )
\qbezier[10](	11.0759	,	68.6142	)(	1.8880	,	61.2804	)(	3.1307	,	 59.1280	 )
\qbezier(	3.1307	,	59.1280	)(	4.3733	,	56.9756	)(	15.3185	,	61.2657	 )
\qbezier(	15.3185	,	61.2657	)(	26.2637	,	65.5557	)(	40.5000	,	73.7750	 )
\qbezier(	40.5000	,	73.7750	)(	54.7363	,	81.9943	)(	63.9241	,	89.3281	 )
\qbezier(	63.9241	,	89.3281	)(	73.1120	,	96.6619	)(	71.8693	,	98.8143	 )

\dashline{2}(60,40)(37.5,78.9711)
\dashline{2}(3.1307	,	59.1280)(71.8693	,	98.8143)
\dashline{2}(60,40)(71.8693	,	98.8143)
\put(60,40){\arc{15}{4.2}{5}}

\put(56,50){$\theta$}
\put(70,65){$r$}
\put(40,50){$h$}
\put(65,35){$o$}
\put(30,82){$u$}
\end{picture}
\hskip 1in
\begin{picture}(100,100)
\setlength{\unitlength}{0.9pt}
\put(60,40){\arc{120}{3.1}{4}}
\put(60,40){\arc{120}{5.4}{6.3}}
  \qbezier[12](	102.4264	,	82.4267	)(	84.8529	,	100.0000	)(	60.0000	,	 100.0000	 )
  \qbezier[12](	60.0000	,	100.0000	)(	35.1471	,	100.0000	)(	17.5736	,	 82.4267	 )

  \qbezier[10](	120.0000	,	40.0000	)(	120.0000	,	48.2844	)(	102.4264	 ,	 54.1422	 )
  \qbezier[10](	102.4264	,	54.1422	)(	84.8529	,	60.0000	)(	60.0000	,	 60.0000	 )
  \qbezier[10](	60.0000	,	60.0000	)(	35.1471	,	60.0000	)(	17.5736	,	 54.1422	 )
  \qbezier[10](	17.5736	,	54.1422	)(	0.0000	,	48.2844	)(	0.0000	,	 40.0000	 )
\qbezier(	0.0000	,	40.0000	)(	0.0000	,	31.7156	)(	17.5736	,	25.8578	 )
\qbezier(	17.5736	,	25.8578	)(	35.1471	,	20.0000	)(	60.0000	,	20.0000	 )
  \qbezier(	60.0000	,	20.0000	)(	84.8529	,	20.0000	)(	102.4264	,	 25.8578	 )
  \qbezier(	102.4264	,	25.8578	)(	120.0000	,	31.7156	)(	120.0000	,	 40.0000	 )

  \qbezier[10](	99.0000	,	85.0000	)(	99.0000	,	87.4853	)(	87.5772	,	 89.2427	 )
  \qbezier[10](	87.5772	,	89.2427	)(	76.1544	,	91.0000	)(	60.0000	,	 91.0000	 )
  \qbezier[10](	60.0000	,	91.0000	)(	43.8456	,	91.0000	)(	32.4228	,	 89.2427	 )
  \qbezier[10](	32.4228	,	89.2427	)(	21.0000	,	87.4853	)(	21.0000	,	 85.0000	 )
\qbezier(	21.0000	,	85.0000	)(	21.0000	,	82.5147	)(	32.4228	,	80.7573	 )
\qbezier(	32.4228	,	80.7573	)(	43.8456	,	79.0000	)(	60.0000	,	79.0000	 )
\qbezier(	60.0000	,	79.0000	)(	76.1544	,	79.0000	)(	87.5772	,	80.7573	 )
\qbezier(	87.5772	,	80.7573	)(	99.0000	,	82.5147	)(	99.0000	,	85.0000	 )

\dashline{2}(60,40)(60,85)
\dashline{2}(60,40)(100,85)
\dashline{2}(20,85)(100,85)
\put(60,40){\arc{15}{4.7}{5.5}}

\put(63,50){$\theta$}
\put(88,58){$r$}
\put(50,57){$h$}
\put(65,35){$o$}
\put(57,89){$u$}
\end{picture}
\end{center}

\begin{lem}\label{trapz} Let $\Gamma_r^h(\mathbb R^n)$ be a trapezoidal ball defined as the above.
Then
$$\dsp\vol{\Gamma_r^h(\mathbb R^n)}
=r\cdot\vol{B_r(\mathbb R^{n-1})}\int_{\theta}^{\pi/2}\sin^n (t)\,dt,$$
where $\theta\in[0,\frac\pi2]$ such that $r\cos\theta=h$.
\end{lem}

\begin{proof}
Let $s=r\cos t\in[0,h]$ be the parameter for the height with the
corresponding angle $t\in[\theta,\frac\pi2]$. Then
\begin{align*}
\vol{\Gamma_r^h(\mathbb R^n)}
&=\int_{0}^{h}\vol{B_{r\sin t}(\mathbb R^{n-1})}ds
=\int_{\theta}^{\pi/2}\vol{B_{r\sin t}(\mathbb R^{n-1})}
r\sin (t)\,dt
\\
&=r\cdot\vol{B_r(\mathbb R^{n-1})}\int_{\theta}^{\pi/2}
\sin^n (t)\,dt.
\end{align*}
\end{proof}

\begin{proof}[{\bf Proof of the volume formula, Lemma \ref{balltube}}]
Because $B_\epsilon(x_i)
\cap B_\epsilon(x_{i+1})\neq\varnothing$ and $B_\epsilon(x_i)
  \cap B_\epsilon(x_j)\cap B_\epsilon(x_k)=\varnothing$ for any $i\neq j\neq k$, we can decompose $\scup{i=1}{N+1}B_\epsilon(x_i)$ as the following (see Figure 0.2): let
\begin{align*}
A^+(x_i)&=\{q\in B_\epsilon(x_i):\;|qx_i|\le|qx_{i+1}|\},
\\
A^-(x_i)&=\{q\in B_\epsilon(x_i):\;|qx_i|\le|qx_{i-1}|\}.
\end{align*}
For $i=2,3,\cdots,N$, let
\begin{align*}
    H^+(x_i)&=A^+(x_i)\cap A^-(x_{i+1})=\{q\in B_\epsilon(x_i)\cap B_\epsilon(x_{i+1}):\;|qx_i|=|qx_{i+1}|\},
\\
    H^-(x_i)&=A^-(x_i)\cap A^+(x_{i-1})=\{q\in B_\epsilon(x_i)\cap B_\epsilon(x_{i-1}):\;|qx_i|=|qx_{i-1}|\};
\end{align*}
and
\begin{align*}
\Gamma^+(x_i)&=\{q\in A^+(x_i)\cap A^-(x_i):\;d(q,H^+(x_i))\le d(q,H^-(x_i))\},
\\
\Gamma^-(x_i)&=\{q\in A^+(x_i)\cap A^-(x_i):\;d(q,H^+(x_i))\ge d(q,H^-(x_i))\}.
\end{align*}
By the construction,
$$\scup{i=1}{N+1}B_\epsilon(x_i)
=A^-(x_1)\cup\left(\scup{i=2}N\Gamma^\pm(x_i)\right)\cup A^+(x_{N+1}).$$
  Note that $H^\pm(x_i)$, $i=2,\cdots,N$ consist of all the possible intersections of any two of $A^-(x_1)$, $\Gamma^\pm(x_i)$, $i=2,\cdots,N$ and $A^+(x_{N+1})$ and $\vol{H^\pm(x_i)}=0$, we have
\begin{align}
\vol{\scup{i=1}{N+1}B_\epsilon(x_i)}
&=\vol{A^-(x_1)}+\vol{A^+(x_{N+1})}
\notag\\
&+\sum_{i=2}^N\vol{\Gamma^+(x_i)}
+\sum_{i=2}^N\vol{\Gamma^-(x_i)}.
\label{balltube.eq1}
\end{align}

  Because $B_\epsilon(x_{i-1})\cup B_\epsilon(x_i)\cup B_\epsilon(x_{i+1})\subset B_{\delta\rho}(x_i)$ which is homeomorphically and $\tau(\delta)$-almost isometrically embedded into $\mathbb R^n$, we have that
\begin{align*}
(1+\tau(\delta))\cdot\vol{\Gamma^\pm(x_i)}
&=\vol{\Gamma_\epsilon^{h_i^\pm}(\mathbb R^n)},
\\
(1+\tau(\delta))\cdot\vol{A^+(x_1)}
&=\frac12\vol{B_\epsilon(\mathbb R^n)}
+\vol{\Gamma_\epsilon^{h_1^+}(\mathbb R^n)},
\\
(1+\tau(\delta))\cdot\vol{A^-(x_{N+1})}
&=\frac12\vol{B_\epsilon(\mathbb R^n)}
+\vol{\Gamma_\epsilon^{h_{N+1}^-}(\mathbb R^n)},
\end{align*}
  where $h_i^+=\frac12|x_ix_{i+1}|$, $h_i^-=\frac12|x_ix_{i-1}|$. Note that it's our convention that the same symbol $\tau(\delta)$ may represent different functions of $\delta$, as long as $\tau(\delta)\to0$ as $\delta\to 0$.
Together with (\ref{balltube.eq1}) and the fact that $h_i^+=h_{i+1}^-$, we get
\begin{align}
(1+\tau(\delta))\cdot
&\vol{\scup{i=1}{N+1}B_\epsilon(x_i)}
=\vol{B_\epsilon(\mathbb R^n)}
+2\sum_{i=1}^N\vol{\Gamma_\epsilon^{h_i^+}(\mathbb R^n)}
\label{balltube.eq2}
\end{align}
  Let $\theta_i\in[0,\frac\pi2]$ such that $\cos\theta_i=h_i^+/\epsilon=\frac{|x_ix_{i+1}|}{2\epsilon}$.  By Lemma \ref{trapz}, we have
\begin{align*}
\vol{\Gamma_\epsilon^{h_i^+}(\mathbb R^n)}
&=\epsilon\cdot\vol{B_\epsilon(\mathbb R^{n-1})}
\int_{\theta_i}^{\pi/2}\sin^n (t)\,dt.
\end{align*}
Plugging this into (\ref{balltube.eq2}), we get (\ref{balltube.e01}).

  To get (\ref{balltube.e02}), we need to write $\int_{\theta_i}^{\pi/2} \sin^n (t)\,dt$ in terms of $|x_ix_{i+1}|$. Let $g(s)=\int_{\theta}^{\pi/2} \sin^n (t)\,dt$, where $\theta\in[0,\frac\pi2]$ with $\cos\theta=\frac{s}{2\epsilon}$. Noting that $\theta=\pi/2$
if and only if $s=0$, we have $g(0)=0$. Further more,
\begin{align*}
g'(s)
&=-\sin^n\theta\cdot \frac{d\theta}{ds}
=-\sin^n\theta\cdot\frac{1}{-2\epsilon\sin\theta}
=\frac{\sin^{n-1}\theta}{2\epsilon};
\\
g''(s)
    &=\frac{1}{2\epsilon}(n-1)\sin^{n-2}\theta\cos\theta \cdot\frac{1}{-2\epsilon\sin\theta}
=\frac{n-1}{-4\epsilon^2}\sin^{n-3}\theta\cos\theta;
\end{align*}
  and thus $g'(0)=\frac{1}{2\epsilon}$, $g''(0)=0$ and $g'''(0)=\frac{c_n}{\epsilon^3}$. The Taylor expansion of $g$ at $s=0$ is
$$g(s)=\int_{\theta}^{\pi/2} \sin^n (t)\,dt
=0+\frac{s}{2\epsilon}+\frac{1}{\epsilon^3}\cdot O(s^3).$$
Let $s=|x_ix_{i+1}|\le \epsilon^2$, we get
\begin{align*}
\int_{\theta_i}^{\pi/2} \sin^n (t)\,dt
&=\frac{1}{2\epsilon}|x_ix_{i+1}|+O(\epsilon)|x_ix_{i+1}|.
\end{align*}
Plugging this into (\ref{balltube.e01}), we get (\ref{balltube.e02}).
\end{proof}

In the rest of this section we assume that $f:Z\to X$ is a distance non-increasing onto map such that $f^{-1}(X^\delta)\subset Z^{\tau(\delta)}$. By Lemma 1.3, $f$ is
homeomorphic on $f^{-1}(X^\delta)$.

\begin{lem}\label{delta.lip} Let the assumptions be as in Theorem A. Let $x,y\in X^\delta$.
For $\delta>0$ sufficiently small, there exists a small constant
$c=c(\rho,\delta)>0$ such that if $|xy|\le c$, then $|f^{-1}(x) f^{-1}(y)|\le2|xy|$.
\end{lem}

\begin{proof}
Assume that $|xy|=\epsilon\ll\delta\rho$ and $|f^{-1}(x)f^{-1}(y)|>2\epsilon$,
consider the metric balls $B_\epsilon(x)$ and $B_\epsilon(y)$. By
Lemma \ref{balltube},
\begin{align*}
&(1+\tau(\delta))\cdot\vol{B_\epsilon(x)\cup B_\epsilon(y)}
\\
&=\vol{B_\epsilon(\mathbb R^n)}
+2 \epsilon\cdot\vol{B_\epsilon(\mathbb R^{n-1})}
\int_{\pi/3}^{\pi/2}
\sin^n (t)\,dt+O(\epsilon^{n+1}).
\end{align*}
Since $B_\epsilon(f^{-1}(x))\cap B_\epsilon(f^{-1}(y))=\varnothing$, we have
\begin{align*}
(1+\tau(\delta))\cdot\vol{B_\epsilon(f^{-1}(x))\cup B_\epsilon(f^{-1}(y))}
=2\vol{B_\epsilon(\mathbb R^n)}.
\end{align*}
Because $f$ is distance non-increasing, $B_\epsilon
(f^{-1}(x))\cup B_\epsilon(f^{-1}(y))\subset f^{-1}(B_\epsilon(x)
\cup B_\epsilon(y))$. Together with that $f^{-1}$ is volume preserving, we get
\begin{align*}
1&=\frac{\vol{f^{-1}(B_\epsilon(x)\cup B_\epsilon(y))}}
{\vol{B_\epsilon(x)\cup B_\epsilon(y)}}
\\
&\geq \frac{(1-\tau(\delta))\cdot 2\vol{B_\epsilon(\mathbb R^n)}}
{\text{vol\,}(B_\epsilon(\mathbb R^n))
+2 \epsilon\cdot\vol{B_\epsilon(\mathbb R^{n-1})}
\int_{\pi/3}^{\pi/2}\sin^n (t)\,dt+O(\epsilon^{n+1})}
\\
&=\frac{(1-\tau(\delta))\cdot 2\int_{0}^{\pi/2}\sin^n (t)\,dt}
{\int_{0}^{\pi/2}\sin^n (t)\,dt
+\int_{\pi/3}^{\pi/2}\sin^n (t)\,dt+O(\epsilon)}.
\quad \text{(see Lemma \ref{trapz}, $\theta=0$)}
\end{align*}
This leads to a contradiction for sufficiently small $\epsilon$ and
$\delta$.
\end{proof}

In the proof of Theorem A, we will need the following result.

\begin{lem} [{[BGP]} 10.6.1] \label{bgp10.6.1} Let $X\in\Alexnk$.
For a fixed sufficiently small $\delta>0$, the union of interior points which do not admit any $(n,\delta)$-strainer has Hausdorff dimension $\le n-2$. In particular, $X^\delta$ is dense.
\end{lem}

\begin{proof}[{\bf Proof of Theorem A}] Since $f$ is distance non-increasing,
it suffices to show that $f$ is distance non-decreasing, i.e. for any $\tilde a, \tilde b\in Z$, $|ab|\ge|\tilde a\tilde b|$, where $a=f(\tilde a)$ and $b=f(\tilde b)$.

For any small $\epsilon_1$, by Lemma \ref{bgp10.6.1}, there are $\tilde a_{\epsilon_1}, \tilde b_{\epsilon_1}
\in Z^{\tau(\delta)}$, $a_{\epsilon_1}=f(\tilde a_{\epsilon_1})$, $b_{\epsilon_1}=
f(\tilde b_{\epsilon_1})\in X^\delta$, such that $|aa_{\epsilon_1}|\le
|\tilde a\tilde a_{\epsilon_1}|<{\epsilon_1}$,
$|bb_{\epsilon_1}|\le|\tilde b\tilde b_{\epsilon_1}|<{\epsilon_1}$.

\vskip2mm

\noindent Case 1. Assume that there exists a minimal geodesic $\geod{a_{\epsilon_1}b_{\epsilon_1}}\subset X$. Then $[a_{\epsilon_1}b_{\epsilon_1}]\subset X^{2\delta}$, because the spaces of directions are isometric along
the interior of a geodesic ([Petrunin 98]). By Lemma \ref{f.homeo} (which will be frequently used
without mentioning), $f^{-1}([a_{\epsilon_1}b_{\epsilon_1}])$ is also a continuous curve. Because $\geod{a_{\epsilon_1} b_{\epsilon_1}}$ is compact, we may let $\rho>0$ such that
$\geod{a_{\epsilon_1} b_{\epsilon_1}}\subset
X^{2\delta}(\rho)$ and $f^{-1}(\geod{a_{\epsilon_1} b_{\epsilon_1}})
\subset Z^{\tau(\delta)}(\rho)$. Let $\{x_i\}_{i=1}^{N+1}$ be an
$\epsilon$-partition of $\geod{a_{\epsilon_1} b_{\epsilon_1}}$,
where $x_1=a_{\epsilon_1}$, $x_{N+1}=b_{\epsilon_1}$. For $\epsilon\ll\delta\rho$.
Because $[a_{\epsilon_1}b_{\epsilon_1}]$ is a geodesic,  Lemma \ref{balltube} can be
applied on the partition $\{x_i\}_{i=1}^{N+1}$. Thus we get
\begin{align*}
&(1+\tau(\delta))\cdot
\vol{\bigcup_{i=1}^{N+1}B_\epsilon(x_i)}
\\
&=\vol{B_\epsilon(\mathbb R^n)}
+\vol{B_\epsilon(\mathbb R^{n-1})}\sum_{i=1}^{N}|x_ix_{i+1}|
+O(\epsilon^{n+1})\sum_{i=1}^{N}|x_ix_{i+1}|
\\
&=\vol{B_\epsilon(\mathbb R^{n-1})}
\cdot|a_{\epsilon_1} b_{\epsilon_1}|+O(\epsilon^n).
\end{align*}
Let $z_i=f^{-1}(x_i)$. By lemma \ref{delta.lip}, $|z_iz_{i+1}|\le 2|x_ix_{i+1}|=2\epsilon$.
  Together with that $f$ is distance non-increasing, one can easily check that $\bigcup_{i=1}^{N+1}
B_\epsilon(z_i)$ satisfies the condition of Lemma \ref{balltube}. Then we have
\begin{align*}
(1+\tau(\delta))\cdot
\vol{\bigcup_{i=1}^{N+1}B_\epsilon(z_i)}
=\vol{B_\epsilon(\mathbb R^{n-1})}
\sum_{i=1}^{N}|z_iz_{i+1}|+O(\epsilon^n).
\end{align*}
Because $f$ is distance non-increasing and volume preserving,
\begin{align*}
1&=\frac{\vol{f^{-1}(\bigcup_{i=1}^{N+1}B_\epsilon(x_i))}}
{\vol{\bigcup_{i=1}^{N+1}B_\epsilon(x_i)}}
\geq \frac{\vol{\bigcup_{i=1}^{N+1}B_\epsilon(z_i)}}
{\vol{\bigcup_{i=1}^{N+1}B_\epsilon(x_i)}}
\\
&=
(1-\tau(\delta))\cdot
\frac{\vol{B_\epsilon(\mathbb R^{n-1})}\sum_{i=1}^{N}|z_iz_{i+1}|+O(\epsilon^n)}
  {\vol{B_\epsilon(\mathbb R^{n-1})}\cdot|a_{\epsilon_1} b_{\epsilon_1}|+O(\epsilon^n)},
\\
&=
(1-\tau(\delta))\cdot
\frac{\sum_{i=1}^{N}|z_iz_{i+1}|+O(\epsilon)}
{|a_{\epsilon_1} b_{\epsilon_1}|+O(\epsilon)}
\\
&\ge (1-\tau(\delta))\cdot
\frac{|\tilde a_{\epsilon_1}\tilde b_{\epsilon_1}|+O(\epsilon)}
{|a_{\epsilon_1} b_{\epsilon_1}|+O(\epsilon)}.
\end{align*}
Let $\epsilon\rightarrow 0$, we get
$$|a_{\epsilon_1} b_{\epsilon_1}|\ge (1-\tau(\delta))|\tilde a_{\epsilon_1}\tilde b_{\epsilon_1}|.$$


\noindent Case 2. Assume that there is no minimal geodesic in $X^\delta$ from $a_{\epsilon_1}$
to $b_{\epsilon_1}$ (since $X$ may not be complete). Because spaces of directions along the interior of geodesic are isometric to each other ([Pet1]), we may assume a curve
$c_1$ in $X^\delta$ from $a_{\epsilon_1}$ to $b_{\epsilon_1}$ such that
$L(c_1)<|a_{\epsilon_1}b_{\epsilon_1}|+\epsilon_1$. Since $c_1(t)$ is
a compact subset in the open set $X^\delta$, we may assume $\eta>0$ such that an $\eta$-tube
of $c_1$ is also contained in $X^\delta$. Consequently, we may assume a piecewise geodesic
$c$ in $X^\delta$ such that $L(c)\le L(c_1)\le |a_{\epsilon_1}b_{\epsilon_1}|+\epsilon_1$. Applying Case 1 to each geodesic segment of $c$, we conclude that
$$|a_{\epsilon_1}b_{\epsilon_1}|\ge L(c)-\epsilon_1\ge (1-\tau(\delta))
|\tilde a_{\epsilon_1}\tilde b_{\epsilon_1}|-\epsilon_1.$$

In either Case 1 or Case 2, we have
\begin{align*}
|ab|&\ge|a_{\epsilon_1} b_{\epsilon_1}|-2\epsilon_1
\ge (1-\tau(\delta))|\tilde a_{\epsilon_1}\tilde b_{\epsilon_1}|-3\epsilon_1
\\&\ge (1-\tau(\delta))\cdot(|\tilde a\tilde b|-2\epsilon_1)-3\epsilon_1.
\end{align*}
Let $\delta\to 0$, $\epsilon_1\to 0$, we get $|ab|\ge|\tilde a\tilde b|$.
\end{proof}

\vskip4mm

\section{Relatively maximum volume}

\vskip4mm
Our proof of the classification part in Theorem B is divided into the following two theorems:
open ball rigidity (Theorem \ref{open.iso}) and isometric involution (Theorem \ref{iso.invol}). Recall that $\tilde o$ denotes the vertex of the cone $\bar C_\kappa^R(\Sigma_p)$ and thus $g\exp_p(\tilde o)=p$.

\begin{thm}\label{open.iso} Under the assumptions of Theorem B, $g\exp_p: C^R_\kappa
(\Sigma)\to B_R(p)$ is an isometry with respect to the intrinsic metrics. In particular, $g\exp_p=\exp_p$.
\end{thm}

By Theorem \ref{open.iso}, $X=\bar C^R_\kappa(\Sigma_p)
/x\sim x'$, where the equivalent relation $x\sim x'$ if and only if $\exp_px=\exp_px'$ and $x,x'\in\Sigma_p\times\{R\}$.

\begin{thm} \label{iso.invol} Let $X=\bar C^R_\kappa(\Sigma_p)
/x\sim x'\in \Alexnk$ defined as the above, then each equivalent class contains at most 2 points. Moreover, the induced involution $\phi: \Sigma_p\times\{R\}
\to \Sigma_p\times\{R\}$, $\phi(x)= x'$ (where $x\sim x'$) is an isometry.
\end{thm}

Recall that the induced gradient-exponential map $g\exp_p: \bar C^R_\kappa(\Sigma)
\to \bar B_R(p)=X$ is distance non-increasing and onto. Indeed, the open ball rigidity is
essentially a consequence of Theorem A and a general property of $\exp_p^{-1}: X\to T_pX$:
$\exp_p^{-1}$ preserves $(n,\delta)$-strained points up to a
constant depending on $\delta$ (see Lemma \ref{exp.delta}). In the proof, let's recall the following
property from [BGP]:

\begin{lem}[{[BGP]} Lemma 7.5 and 11.2] \label{bgp75} Let $p\in X\in\Alexnk$. Then for any
$\delta>0$, there is a small neighborhood $U_p$ of $p$ such that for any triangle $\triangle pab$
with $a,b\in U_p$ each angle of $\triangle pab\subset X$ differs from the comparison angle of $\tilde
\triangle pab\subset \mathbb S_{\kappa}^2$ by less than $\delta$.
\end{lem}

\begin{lem}\label{exp.delta} Let $q\in X^\delta$. Then for any $p\in X$, $\drn_p^q\in
\Sigma_p^{\tau(\delta)}$. Consequently, $\exp_p^{-1}(q)\in \bar C_\kappa^R(\Sigma_p)^{\tau(\delta)}$.
\end{lem}

\begin{proof} Since $q\in X^\delta$, by Lemma \ref{bgp9.4}, we may assume an $(n,2\delta)$-strainer $\{(a_i, b_i)\}$
for $q_1\in\geod{pq}$ and near $q$, such that $b_n=q$, $a_n\in\geod{pq_1}$. Because the spaces of directions are isometric along the interior of a geodesic ([Petrunin 98]), there is $q'\in\geod{pq}\cap U_p$ which has an $(n,\tau(\delta))$-strainer
$\{(a'_i, b'_i)\}$. By the same reason as the above, we can assume that $a_n'\in\geod{pq'}$ and $b_n'\in \geod{q'q}$.

In addition, we can assume that $|q'a'_i|$, $|q'b'_i|$ are short so that $a'_i, b'_i\in U_p$
and $\measuredangle a'_ipq', \measuredangle b'_ipq'<5\delta$. We claim that $\{(\drn_p^{a'_i},
\drn_p^{b'_i})\}_{i=1}^{n-1}$ forms an $(n-1,\tau(\delta))$-strainer at $\drn_p^q\in \Sigma_p$.
It's easy to see that
$\measuredangle a'_ipq' =\tilde\measuredangle a'_ipq'+\tau(\delta)=\frac{|a'_iq'|}{|pq'|}+\tau(\delta)$.
Thus
$$\cos\tilde\measuredangle\drn_p^{a'_i}\drn_p^{q'}\drn_p^{x_j}
=\frac{|a'_iq'|^2+|x_jq'|^2-|a'_ix_j|}{2|a'_iq'||x_jq|} +\tau(\delta)
=\cos\tilde\measuredangle a'_iq'x_j +\tau(\delta),$$
where $i,j=1,2,\cdots,n-1$, $x_j=a'_j$ or $b'_j$.
\end{proof}

To conclude the open ball rigidity by applying Theorem A, we need to check that $g\exp_p^{-1}(X^\delta)\subseteq \bar C_\kappa^R(\Sigma_p)^{\tau(\delta)}$. We obtain this by showing
$g\exp_p=\exp_p$ when $\text{vol}(X)=v(\Sigma_p,\kappa,R)$.

\begin{lem}\label{exp} If
$\vol{B_R(p)}=\vol{C_{\kappa}^R(\Sigma_p)}$, then the gradient exponential map
is actually an exponential map $\exp_p:\bar C_{\kappa}^R(\Sigma_p)
\rightarrow \bar B_R(p)$ which preserves the distance along the radio direction.
\end{lem}

\begin{proof}
Clearly, the map $\exp^{-1}_p:\bar B_R(p)\to
\bar C_{\kappa}^R(\Sigma_p)$ (If there is more than one image, we will pick one)
is distance non-decreasing. Because
  $$\vol {C_{\kappa}^R(\Sigma_p)}=\vol X\le \vol{\exp_p^{-1}(X)}\le \vol {C_{\kappa}^R(\Sigma_p)},$$
  $\exp_p^{-1}(X)$ is dense in ${C_{\kappa}^R(\Sigma_p)}$. For any $z\in {C_{\kappa}^R(\Sigma_p)}$,
  there is a sequence $x_i\in X$, such that $\exp_p^{-1}(x_i)=z_i\to z$. Let $\exp_p:{C_{\kappa}^R
  (\Sigma_p)}\to X$; $\exp_p(z)=\dsp\lim_{i\to\infty} x_i$. Such $\exp_p$ is well defined,
since if there is another sequence $\exp_p^{-1}(x_i')=z_i'\to z$, then
$$d(\lim_{i\to\infty}x_i,\lim_{i\to\infty}x_i')
=\lim_{i\to\infty}d(x_i,x_i')
\le\lim_{i\to\infty}d(z_i,z_i')
=0.
$$
It's clear that $\exp_p$, defined as an extension of $\exp_p^{-1}$, is distance
non-increasing. Moreover, it preserves the distance along the radio direction.

We now show that any geodesic from $p=\exp_p(\tilde o)$ to $q=\exp_p(\tilde q)\in
B_R(p)$ can be extended. Therefore $\exp_p$ is a bijection since geodesic does not
bifurcate. Let $\geod{\tilde o\tilde q}$ be the geodesic in $C_{\kappa}^R(\Sigma_p)$
such that $\exp_p(\geod{\tilde o\tilde q})=\geod{pq}$ and $\tilde q'\in C_{\kappa}^R
(\Sigma_p)$ be the extended point of $\geod{\tilde o\tilde q}$. Then
$$|pq|+|qq'|\leq |\tilde o\tilde{q}|+|\tilde q\tilde{q}'|
=|\tilde o\tilde{q}'|=|pq'|,
$$ which forces $\geod{pq}\cup\geod{qq'}$ being a
geodesic.
\end{proof}

\begin{proof} [{\bf Proof of Theorem \ref{open.iso}}] For $X\in \mathcal A^R_\kappa(\Sigma)$ with $\text{vol}(X)=
v(\Sigma,\kappa,R)$, by Lemma \ref{exp.delta} and Lemma \ref{exp} we see that $\exp_p: C^R_\kappa(\Sigma)\to B_R(p)$
is a distance non-increasing onto map which satisfies the assumptions in Theorem A (note that $\exp_p: \bar C^R_
\kappa(\Sigma_p)\to \bar B_R(p)=X$ may not satisfy the assumption of Theorem A).
\end{proof}



In the proof of Theorem \ref{iso.invol}, our main technique lemma is Lemma \ref{non-cross}. Let $\phi:\Sigma\times\{R\}\to \Sigma\times\{R\}$ be defined as in Theorem \ref{iso.invol}. We first observe that $\phi$ is an involution. Let $L_p(X)=\exp_p(\Sigma\times\{R\})=\{x\in X: |px|=R\}$.

\begin{lem}\label{invol}

Let $X=\bar C_{\kappa}^R(\Sigma)/x\sim x'\in\Alexnk$ defined as in Theorem \ref{iso.invol}. For any $q\in L_p(X)$, if $\tilde q_1\neq\tilde q_2$ with $\exp_p(\tilde q_1)=\exp_p(\tilde q_2)=q$, then the
loop $\exp_p(\geod{\tilde o\tilde q_1})\cup \exp_p(\geod{\tilde o\tilde q_2})$
forms a local geodesic at $q$. Consequently, $\exp^{-1}_p(q)$ contains at most 2 points.

\end{lem}

\begin{proof}
It's clear that $\exp_p(\geod{\tilde o\tilde q_i})$ are minimal geodesics,
$i=1,2$. Let $x_i\in X$ be a point on $\exp_p(\geod{\tilde o \tilde q_i})$
and $\tilde x_i=\exp_p^{-1}(x_i)$, $i=1,2$. We claim that if $x_1,x_2$ are both
close to $q$ enough, the geodesic $\geod{x_1x_2}$ intersects with
$L_p(X)$. If not, then $\geod{x_1x_2}\subset B_R(p)$. By the
assumption, $|x_1x_2|_X=|\tilde x_1\tilde x_2|_{\bar
C_{\kappa}^R(\Sigma)}$. Let $x_1, x_2\to q$. We get that $|x_1x_2|_X\to 0$ and
$|\tilde x_1\tilde x_2|_{\bar C_{\kappa}^R(\Sigma)}\to |\tilde q_1\tilde q_2|_{\bar
C_{\kappa}^R(\Sigma)}>0$, a contradiction.

Let $a\in \geod{x_1x_2}\cap L_p(X)$, it remains to show that $a=q$. For $i=1,2$,
$$|x_ia|\ge|pa|-|px_i|=|pq|-|px_i|=|x_iq|.$$
Thus
$$|x_1q|+|x_2q|\le |x_1a|+|x_2a|=|x_1x_2|,$$
which forces both of the above inequalities to be equalities, and thus $a=q$.
\end{proof}

As a corollary of Lemma \ref{invol}, we conclude that for $X\in \mathcal A^R_\kappa(\Sigma)$, $\kappa>0$ and $\frac{\pi}{2\sqrt \kappa}<R<\frac{\pi}{\sqrt \kappa}$, $\text{vol}(C_\kappa^R(\Sigma))$ is not the optimal upper bound for $\text{vol}(X)$ (c.f. [GP]). Equivalently, we have

\begin{cor} \label{pso.non.opt}  Assume $X\in\mathcal A_{\kappa}^R(\Sigma)$ with $\text{vol}(X)
=\text{vol}(\bar C_\kappa^R(\Sigma))$ and $\kappa>0$, then
$R\leq\frac{\pi}{2\sqrt \kappa}$ or $R=\frac{\pi}{\sqrt \kappa}$. In the
second case, $X=C_\kappa(\Sigma)$ which is the $k$-suspension of
$\Sigma$.
\end{cor}

\begin{proof}
Assume $\frac{\pi}{2\sqrt \kappa}<R<\frac{\pi}{\sqrt \kappa}$. Let $p\in X$ such that $\Sigma_p=\Sigma$. It's clear that $\text{rad}_p(X)=R$.
We claim that
$L_p(X)=\{q\}$ has only one point. Then by Lemma \ref{invol},
$\Sigma_p\times\{R\}=\exp_p^{-1}(q)$ contains at most 2 points, a
contradiction. Let $a\neq b\in L_p(X)$, consider the triangle $\triangle pab$
and the compared triangle $\widetilde{\triangle} pab\in S^2_\kappa$. Take
$c\in \geod{ab}$ and the corresponding $\tilde{c}\in
\geod{\tilde{a}\tilde{b}}$ with $|ac|=|\widetilde{a}\widetilde{c}|$. By
the triangle comparison, $|pc|\geq |\tilde{p}\tilde{c}|>R$, a contradiction.

Note that the case $R=\frac \pi{\sqrt\kappa}$ follows from Theorem \ref{open.iso}.
\end{proof}

It remains to show that $\phi$ is an isometry. The following lemma plays an important role in the study of the angles in the gluing space $X$.

\begin{lem}\label{cone.angle} Let $a,b\in C_\kappa(\Sigma)$. Then $\measuredangle apb=\tilde
\measuredangle apb$ and $\measuredangle pab=\tilde \measuredangle pab$.
\end{lem}

\begin{proof}

The proofs are essentially same for different $\kappa$. For simplicity, we only give a proof for $\kappa=0$. Note that $\measuredangle apb=\tilde \measuredangle apb$ by the definition of $C_\kappa(\Sigma)$.

To see $\measuredangle pab=\tilde \measuredangle pab$, shortly extend
the geodesic $\geod{pa}$ to $a'$ and apply the cosine law to the
triangles $\triangle aa'b$, $\triangle pa'b$ and $\triangle pab$. We get

\begin{align}
|a'b|^2&=|aa'|^2+|ab|^2-2|aa'||ab|\cos\widetilde{\measuredangle}a'ab,
\label{cone.angle.e1}\\
|a'b|^2&=|pa'|^2+|pb|^2-2|pa'||pb|\cos\measuredangle apb
\label{cone.angle.e2}\\
&=(|pa|+|aa'|)^2+|pb|^2-2(|pa|+|aa'|)|pb|\cos\measuredangle apb,
\notag
\\
|ab|^2&=|pa|^2+|pb|^2-2|pa||pb|\cos\measuredangle apb.
\label{cone.angle.e3}
\end{align}
Calculating $(\ref{cone.angle.e1})+(\ref{cone.angle.e3})-(\ref{cone.angle.e2})$, we get
\begin{align*}
0&=|ab|\cos\widetilde{\measuredangle}a'ab
+|pa|-|pb|\cos\measuredangle apb
\\
&\ge |ab|\cos\measuredangle a'ab
+|pa|-|pb|\cos\measuredangle apb
\\
&=-|ab|\cos\measuredangle pab
+|pa|-|pb|\cos\measuredangle apb.
\end{align*}
Since $\measuredangle pab\ge\tilde\measuredangle pab$ and
$\measuredangle apb=\tilde\measuredangle apb$, the above inequality implies
\begin{align*}
|pa|&\le |ab|\cos\measuredangle pab +|pb|\cos\measuredangle apb
\\
&\le |ab|\cos\tilde\measuredangle pab +|pb|\cos\tilde\measuredangle apb
=|pa|,
\end{align*}
which forces $\measuredangle pab=\tilde\measuredangle pab$.
\end{proof}

\begin{cor} \label{stay.bdy}
  Let $x,y\in X$ be two points. If $\geod{xy}\cap L_p(X)\neq\varnothing$, then either $\geod{xy}\subset L_p(X)$ or $\geod{xy}\cap L_p(X)$ is finite.
\end{cor}

\begin{proof}
  Let $x\notin L_p(X)$, we shall show that $\geod{xy}\cap L_p(X)$ is finite. Let $a\in\geod{xy}\cap L_p(X)$ be the accumulation point which is closest to $x$. Clearly $a\neq x$ since $x\notin L_p(X)$. Thus there is a geodesic segment $\geod{ba}$ of $\geod{xy}$ with that $\geod{ba}-\{a\}\subset B_R(p)$. Since $|pb|<|pa|=R$, by Lemma \ref{cone.angle},
$$\measuredangle pab=\tilde\measuredangle pab<\frac\pi2.$$
  On the other hand, because there are $a_i\in \geod{xy}\cap L_p(X)$ with $a_i\to a$ as $i\to\infty$ and $|pa|=|pa_i|=R$, by the first variation formula, we get
$$\measuredangle pay=\frac\pi2.$$
Therefore $\pi=\measuredangle pab+\measuredangle pay<\pi$, a contradiction.
\end{proof}

As another corollary, we prove Theorem \ref{iso.invol} for the special case $\kappa>0$ and $R=\frac\pi{2\sqrt\kappa}$.

\begin{cor} \label{spec.iso}
  Theorem \ref{iso.invol} holds for the case $\kappa>0$ and $R=\frac\pi{2\sqrt\kappa}$.
\end{cor}

\begin{proof}

Let $x,y\in L_p(X)$, $\tilde x_1, \tilde x_2, \tilde y_1, \tilde y_2\in \Sigma\times\{R\}$ with $\exp_p(\tilde x_1)=\exp_p(\tilde x_2)=x$, $\exp_p(\tilde y_1)=\exp_p(\tilde y_2)=y$. We will show that $|\tilde x_1\tilde y_1|_{\bar C_{\kappa}^R(\Sigma)}=|\tilde x_2 \tilde y_2|_{\bar C_{\kappa}^R(\Sigma)}$. Assume $|\tilde x_1\tilde y_1|_{\bar C_{\kappa}^R (\Sigma)}>|\tilde x_2\tilde y_2|_{\bar C_{\kappa}^R(\Sigma)}$. Then there is a point $a\notin L_p(X)$ (take $\exp_p^{-1}(a)$ be close to $x_1$) such that $\geod{ay}\cap L_p(X)$ contains a point $b\neq y$. Because $\exp_p$ is distance non-increasing and $\Sigma\times\{\frac\pi{2\sqrt\kappa}\}$ is totally geodesic, $\geod{by}\subset L_p(X)$, which contradicts to Corollary \ref{stay.bdy}.
\end{proof}

Let Fix$(\phi)=\{\tilde x\in \Sigma\times\{R\}: \phi(\tilde x)=\tilde x\}$ be the fixed points set. Let $L_p^1(X)=\exp_p(\text{Fix}(\phi))$ denote the image. Due to Lemma \ref{invol}, let $L_p^2(X)=L_p(X)-L_p^1(X)$ denote the points that are identified from exactly two points, i.e. for any $x\in L_p^2(X)$, $\exp^{-1}(x)=\{\tilde x^+, \tilde x^-\}$ contains exactly two points.

In the rest proof of Theorem \ref{iso.invol}, by Corollary \ref{stay.bdy}, \ref{spec.iso} and their proofs,
we can always assume $R<\frac{\pi}{2\sqrt\kappa}$ for $\kappa>0$ and that for any $x,y\in X$, $\geod{xy}\cap L_p(X)$ is finite if it is not empty. More over, the following corollary shows that $\,]\,xy\,[\,\cap L_p(X)\subset L_p^2(X)$, where $\,]\,xy\,[\,$ denotes the geodesic connecting $x,y$ without the end points.

\begin{cor}\label{double-cross}
  Let the assumption be as in Theorem \ref{iso.invol}. Assume $R<\frac{\pi}{2\sqrt\kappa}$ when $\kappa>0$.  For any $x,y\in X$, if $q\in\,]\,xy\,[\,\cap L_p(X)$, then $q\in L_p^2(X)$.
\end{cor}

\begin{proof}
  Not losing generality, assume $x, y\notin L_p(X)$ and $\,]\,xy\,[\,\cap L_p(X)=\{q\}$. If $q\in L_p^1(X)$,
  by Lemma \ref{cone.angle}, $\measuredangle xqp=\tilde \measuredangle xqp<\frac\pi2$ and
$\measuredangle yqp=\tilde \measuredangle yqp<\frac\pi2$. Thus $\measuredangle xqp+
  \measuredangle yqp<\pi$, which contradicts to the fact that $\geod{xy}$ is a geodesic.
\end{proof}

Now we are ready to prove our main technique lemma. Let $x\in L_p^2(X)$ and $\{\tilde x^+, \tilde x^-\}=\exp_p^{-1}(x)$ denote the pre-image. Then there are exactly two geodesics $\exp_p(\geod{\tilde o\tilde x^+})$, $\exp_p(\geod{\tilde o\tilde x^-})$ connecting $x$ to $p$. To distinguish geodesics and angles, we use the following notation.
\begin{itemize}
  \item Let $\geod{px^+}$ and $\geod{px^-}$ denote $\exp_p(\geod{\tilde o\tilde x^+})$ and $\exp_p(\geod{\tilde o\tilde x^-})$ respectively.
\end{itemize}
In addition, for $y\in L_p^2(X)$ and $\exp_p^{-1}(y)=\{\tilde y^+, \tilde y^-\}$,
\begin{itemize}
  \item Let $\geod{x^\pm y^\pm}$ denote $\exp_p(\geod{\tilde x^\pm \tilde y^\pm})$;
  \item Let $|x^\pm y^\pm|$ denote the length of the geodesics $\geod{x^\pm y^\pm}$;
  \item Let $\measuredangle x^\pm py^\pm$ denote the angle between $\geod{px^\pm}$ and $\geod{py^\pm}$ at $p$;
  \item Let $\measuredangle px^\pm y^\pm$ denote the angle between $\geod{px^\pm}$ and
$\geod{x^\pm y^\pm}$ at $x$.
\end{itemize}


\begin{lem}\label{non-cross} Let the assumption be as in Theorem \ref{iso.invol}. Assume $R<\frac{\pi}{2\sqrt\kappa}$ when $\kappa>0$. Then for any $\tilde x\neq\tilde y\in \Sigma\times\{R\}$ with $|\tilde x\tilde y|$ sufficiently small,
  $$\left|\frac{|\phi(\tilde x)\phi(\tilde y)|}{|\tilde x\tilde y|}-1\right|\le20|\tilde x\tilde y|.
$$
\end{lem}



\begin{proof} 
  For simplicity, we give a proof for the case $\kappa=0$. The other cases can be carried out similarly.
  Throughout the proof, we will frequently use Lemma \ref{invol}, \ref{cone.angle} and Corollary \ref{double-cross} without mentioning. We will also assume that for any $a,b\in X$, $\geod{ab}\cap L_p(X)$ is finite if it is not empty.

  Clearly, $\phi$ preserves the distance when $x$ and $y$ are both in $L_p^1(X)$. Let $x\in L_p^2(X)$, $y\in L_p(X)$ (if $y\in L_p^1(X)$, $\tilde y^+=\tilde y^-$ will denote the same point and the argument will still go through). Because $\geod{xy}\cap L_p(X)$ is finite, not losing generality,
  assume $\geod{xy}=\geod{x^-y^-}$. Thus $\measuredangle x^-py^-\le \measuredangle x^+py^+$.
  Let $\beta_0=\measuredangle x^-py^-$. Since $|x^-y^-|=2R\sin\frac{\beta_0}2$ and $|x^+y^+|=2R\sin\frac{\measuredangle x^+py^+}2$, it's sufficient to show that
\begin{equation} \label{almost.eq}
10\beta_0^2+\beta_0\ge \measuredangle x^+py^+.
\end{equation}

\begin{center}
\begin{picture}(150,170)
\put(58, 0){Figure 3}
\put(75,150){\circle*{1}}
    \qbezier(	20.0000	,	70.0000	)(	20.0000	,	55.5022	)(	36.1091	,	 45.2511	 )
\qbezier(	36.1091	,	45.2511	)(	52.2182	,	35.0000	)(	75.0000	,	35.0000	 )
\qbezier(	75.0000	,	35.0000	)(	97.7818	,	35.0000	)(	113.8909	,	 45.2511	 )
\qbezier(	113.8909	,	45.2511	)(	130.0000	,	55.5022	)(	130.0000	,	 70.0000	 )

\put(40,43.5){\circle*{1}} 
\put(127.5,60){\circle*{1}} 
\dashline{2}(40,43.5)(127.5,60) 
\qbezier(75,150)(40,95)(40,43.5)
\qbezier[30](75,150)(65,85)(40,43.5)
\qbezier[30](75,150)(120,95)(127.5,60)
\put(54,38){\circle*{1}} 
\put(113,45){\circle*{1}} 
\put(40.5,53){\circle*{1}} 
\drawline(40.5,53)(54,38) 
\dashline{2}(54,38)(113,45) 
\drawline(113,45)(127.5,60) 
\qbezier(75,150)(52,85)(54,38)
\qbezier[30](75,150)(72,85)(54,38)
\qbezier[30](75,150)(108,95)(113,45)
\put(54,48.5){\circle*{1}} 
\put(70,35.5){\circle*{1}} 
\drawline(54,48.5)(70,35.5) 

\put(73,160){$p$}
\put(35,34){$x$}
\put(29,58){$u_0$}
\put(50,25){$a_1$}
\put(110,30){$b_1$}
\put(126,48){$y$}
\put(58,51){$u_1$}
\put(68,23){$a_2$}

\put(66,130){\arc{10}{1.4}{2.6}}
\put(77,140){\arc{15}{0.6}{2}}
\put(48,122){$\alpha_1$}
\put(75.5,122){$\beta_1$}
\put(43,65){$\scriptstyle [+]$}
\put(76,85){$\scriptstyle [-]$}
\end{picture}
\end{center}

  Take $u_0\in \geod{px^+}$ with $|u_0x^+|=\epsilon$. Let $\geod{u_0y}$ be a geodesic. If $\geod{(u_0y)}\cap L_p(X)\neq\varnothing$, let $a_1(\neq y)$ and $b_1$ ($b_1$ can be $y$) be the first and second intersection points in $\geod{u_0y}\cap L_p(X)$ along the direction $\drn_{u_0}^y$ (see Figure 3). Assign $\pm$ to $\exp_p^{-1}(a_1), \exp_p^{-1}(b_1)$ such that $\measuredangle pa_1^+u_0<\frac\pi2$. Let $\alpha_1=\measuredangle x^+pa_1^+$ and $\beta_1=\measuredangle a_1^-pb_1^-$.
  In the case of $\geod{(u_0y)}\cap L_p(X)=\varnothing$, we take $b_1=a_1=y$ and $\beta_1=0$.

  Because $\geod{u_0a_1^+}*\geod{a_1^-b_1^-}*\geod{b_1^+y}$ is a minimal geodesic, by triangle inequality,
$$|u_0x|+|xy|\ge|u_0a_1^+|+|a_1^-b_1^-|+|b_1y|.$$
This implies
\begin{equation}
\epsilon+2R\sin\frac{\beta_0}2
\ge |u_0a_1^+|+2R\sin\frac{\beta_1}2.
\end{equation}

  Applying the cosine law (the form in Lemma \ref{t4.1.9} (5)) in $\triangle pu_0a_1$ with the angle $\measuredangle u_0pa_1^+=\alpha_1$, we get that
$$|u_0a_1^+|=\sqrt{\epsilon^2+4R(R-\epsilon)\sin^2\frac{\alpha_1}2}
\ge 2(R-\epsilon)\sin\frac{\alpha_1}2.$$
Thus
\begin{equation}
\epsilon+2R\sin\frac{\beta_0}2
\ge 2(R-\epsilon)\sin\frac{\alpha_1}2
+2R\sin\frac{\beta_1}2.
\end{equation}


  If $\geod{(u_0y)}\cap L_p(X)=\varnothing$, we stop here. If $\geod{(u_0y)}\cap L_p(X)\neq\varnothing$, we proceed with $u_1\in\geod{pa^+_1}$ and $|u_1a_1|=\epsilon$. Let $\geod{u_1b_1}$ be a geodesic. Again, if $\geod{(u_1b_1)}\cap L_p(X)\neq\varnothing$, Let $a_2(\neq y)$ and $b_2$ (can be $b_1$) be the first and second intersection points in $\geod{u_1b_1}\cap L_p(X)$ along the direction $\drn_{u_1}^{b_1}$. Assign $\pm$ to $\exp_p^{-1}(a_2), \exp_p^{-1}(b_2)$ such that $\measuredangle pa_2^+u_1<\frac\pi2$. Let $\alpha_2=\measuredangle a_1^+pa_2^+$ and
  $\beta_2=\measuredangle a_2^-pb_2^-$. If $\geod{(u_1b_1)}\cap L_p(X)=\varnothing$, then $a_2=b_2=b_1$, $\beta_2=0$ and we stop the process. Proceed inductively until $\geod{(u_Nb_N)}\cap L_p(X)=\varnothing$, which yields that $a_{N+1}=b_{N+1}=b_N$ and $\beta_{N+1}=0$. We claim that $N$ is finite, and moreover,
\begin{equation} \label{non-cross.eN}
(N+1)\epsilon< 5R\cdot\beta_0^2.
\end{equation}
For each $0\le i\le N$, we have
\begin{equation} \label{e.iso.induct.length}
\epsilon+2R\sin\frac{\beta_i}2
\ge |u_ia_{i+1}^+|+2R\sin\frac{\beta_{i+1}}2,
\end{equation}
and
\begin{equation} \label{e.iso.induct.angle}
\epsilon+2R\sin\frac{\beta_i}2
\ge 2(R-\epsilon)\sin\frac{\alpha_{i+1}}2
+2R\sin\frac{\beta_{i+1}}2,
\end{equation}
  where $\alpha_i=\measuredangle a_i^+pa_{i+1}^+$, $\beta_i=\measuredangle a_i^-pb_i^-$. Summing up (\ref{e.iso.induct.angle}) for $i=0,1,\cdots,N$ and applying (\ref{non-cross.eN}), we get
\begin{align*}
5R\cdot\beta_0^2+2R\sin\frac{\beta_0}2
&\ge
(N+1)\epsilon+2R\sin\frac{\beta_0}2
\\
&\ge 2(R-\epsilon)\sum_{i=1}^N\sin\frac{\alpha_i}2
\ge 2(R-\epsilon)\sin\frac{\sum_{i=1}^N\alpha_i}2
\\
&\ge 2(R-\epsilon)\sin\frac{\measuredangle x^+pb_N}2.
\end{align*}
Since $b_N\to b_1\to y^+$ when taking $\epsilon\to0$, (\ref{almost.eq}) follows.

  It remains to show (\ref{non-cross.eN}). A sum of (\ref{e.iso.induct.length}) for $i=0,1,\cdots,N$ indicates that the upper bound of $N$ relies on an estimate of $|u_ia^+_{i+1}|$ in terms of $\epsilon$ and $\beta_{i+1}$. Note that $a_{i+1}=\geod{u_ib_{i+1}}\cap\left(\geod{pa^+_{i+1}}*\geod{pa^-_{i+1}}\right)$ and $\geod{pa^+_{i+1}}*\geod{pa^-_{i+1}}$ is a local geodesic at $a_{i+1}$, we have $\measuredangle pa^+_{i+1}u_i
  =\measuredangle pa^-_{i+1}b_{i+1}=\frac\pi2-\frac{\beta_{i+1}}2$. Applying the cosine law in triangle
$\triangle pu_ia_{i+1}^+$, we get
$$(R-\epsilon)^2=R^2+|u_ia_{i+1}^+|^2-2R|u_ia_{i+1}^+|\sin\frac{\beta_{i+1}}2,
$$
i.e.
$$|u_ia_{i+1}^+|^2-2R\sin\frac{\beta_i}2\cdot |u_ia_{i+1}^+|+R\epsilon-\epsilon^2=0.
$$
  Solving for $|u_ia_{i+1}^+|$ and taking in account that $\epsilon>0$ is small, we have
\begin{align*}
|u_ia_{i+1}^+|
&\ge R\sin\frac{\beta_{i+1}}2
-\sqrt{\left(R\sin\frac{\beta_{i+1}}2\right)^2-(R\epsilon-\epsilon^2)}
>\frac\epsilon{4\sin\frac{\beta_{i+1}}{2}}.
\end{align*}
  Note that $\beta_i$ is decreasing, which is implied by (\ref{e.iso.induct.length}) and $|u_ia_{i+1}^+|>|u_ia_i^+|=\epsilon$. We get
\begin{equation}
|u_ia_{i+1}^+|>\frac\epsilon{4\sin\frac{\beta_0}{2}}.
\label{non-cross.eq26}
\end{equation}
Plugging (\ref{non-cross.eq26}) into (\ref{e.iso.induct.length}), we get
\begin{equation}\label{non-cross.eq27}
\epsilon+2R\sin\frac{\beta_i}2
>\frac\epsilon{4\sin\frac{\beta_0}{2}}+2R\sin\frac{\beta_{i+1}}2.
\end{equation}
Summing up (\ref{non-cross.eq27}) for $i=0,1,\cdots,N$, we get
$$(N+1)\epsilon+2R\sin\frac{\beta_0}2
> (N+1)\frac{\epsilon}{4\sin\frac{\beta_0}{2}}.
$$
Therefore
  $$(N+1)\epsilon< \frac{8R\sin^2\frac{\beta_0}2}{1-4\sin\frac{\beta_0}2}< 5R\cdot\beta_0^2.$$
\end{proof}


\begin{proof}[{\bf Proof of Theorem \ref{iso.invol}} (Assuming $R<\frac{\pi}{2\sqrt\kappa}$ when $\kappa>0$)] By Lemma \ref{non-cross}, $\phi$ is a continuous
involution and thus a homeomorphism. It reduces to show that $\phi: \Sigma\times \{R\}\to
\Sigma\times\{R\}$ preserves length of any curve $c: [0,1]\to \Sigma\times \{R\}$.
Given $\delta, \epsilon>0$, we may assume a partition $P: 0=t_0<t_1<\cdots <t_N=1$
with $|c(t_i)c(t_{i+1})|\le \delta$ such that the length of the curves satisfy
$$L(c)<\sum_{i=0}^{N-1}|c(t_i)c(t_{i+1})|+\frac \epsilon 2,\qquad L(\phi(c))<\sum_{i=0}^{N-1}|\phi(c(t_i))
\phi(c(t_{i+1}))|+\frac \epsilon2.$$
Then
\begin{align*}
|L(c)-L(\phi(c))|&\le \sum_{i=0}^{N-1}||c(t_i)c(t_{i+1})|-|\phi(c(t_i))
\phi(c(t_{i+1}))||+\epsilon\\&\le \sum_{i=0}^{N-1}20|c(t_i)c(t_{i+1})|^2+\epsilon\\
&\le 20\delta\cdot \sum_{i=0}^{N-1}|c(t_i)c(t_{i+1})|+\epsilon\\&\le
20\delta\cdot L(c)+\epsilon.
\end{align*}
Since $\epsilon>0, \delta>0$ can be chosen arbitrarily small, we conclude the desired result.
\end{proof}

\begin{proof}[\bf Completion of Proof of Theorem B] By Theorems \ref{open.iso} and \ref{iso.invol}, we identify $X$ with $\bar C_{\kappa}^R(\Sigma_p)/x\sim\phi(x)$. We shall show that the metric on $X$ coincides with the metric induced from the identification $x\sim\phi(x)$. It's equivalent to show that $\exp_p:\bar C_{\kappa}^R(\Sigma_p)\to X$ preserves lengths of geodesics. Let $\gamma\subset \bar C_{\kappa}^R(\Sigma_p)$ be a geodesic and $\sigma=f(\gamma)$. Since $L(\gamma)\ge L(\sigma)$, it remains to show that $L(\sigma)\ge L(\gamma)$. Because either $\gamma\subset\Sigma\times\{R\}$ or $\gamma\cap(\Sigma\times\{R\})$ has at most 2 points, we only need to check for the case $\gamma\subset\Sigma\times\{R\}$ i.e., $\sigma\subset L_p(X)$. For any $\epsilon>0$, let $\{x_i\}_{i=0}^{2N+1}\subset\sigma$ be an $\epsilon$-partition and
$$L(\sigma)=\lim_{\epsilon\to 0}\sum_{i=0}^{2N}|x_ix_{i+1}|.$$
Let $a_i\in\gamma$ so that $\exp_p(a_i)=x_i$. Choose $b_{2k}\in C_{\kappa}^R(\Sigma)$, $k=0,1,\cdots,N$, with $|a_{2k}-b_{2k}|<\epsilon^4$. Let $b_{2k+1}=a_{2k+1}$ for $k=0,1,\cdots,N$ and $y_i=\exp_p(b_i)$ for $i=0,1,\cdots 2N+1$. Then $|y_i-x_i|\le|b_i-a_i|<\epsilon^4$ and thus
$$L(\sigma)=\lim_{\epsilon\to 0}\sum_{i=0}^{2N}|y_iy_{i+1}|.$$
We claim that $\geod{y_iy_{i+1}}\cap L_p(X)$ is either $y_i$ or $y_{i+1}$. By Corollary \ref{stay.bdy}, let $u,v\in\geod{y_iy_{i+1}}\cap L_p(X)$ and there is no crossing point in between. Not losing generality, assume $y_i\notin L_p(X)$ and $|y_iu|<|y_iv|$. Let $\geod{u^-v^-}\subset\geod{y_iy_{i+1}}$. Because the involution $\phi$ is an isometry (Theorem \ref{iso.invol}), $L(\geod{u^+v^+})=L(\geod{u^-v^-})$. Thus $\geod{y_iu}\cup \geod{u^+v^+}\neq \geod{y_iu}\cup \geod{u^-v^-}$ is also a geodesic, which yields a bifurcation of geodesics.

By the claimed property, we have that $|y_iy_{i+1}|=|b_ib_{i+1}|$. Since $\sum_{i=0}^{2N}|b_ib_{i+1}|
\ge L(\gamma)$, we have
\begin{align*}
L(\sigma)=\lim_{\epsilon\to 0}\sum_{i=0}^{2N}|y_iy_{i+1}|
=\lim_{\epsilon\to 0}\sum_{i=0}^{2N}|b_ib_{i+1}|
\ge L(\gamma).
\end{align*}

It remains to show that for $\Sigma\in \text{Alex}^{n-1}(1)$, if $\phi: \Sigma\times
\{R\}\to \Sigma\times \{R\}$ is an isometric involution, then $X=\bar C^R_\kappa
(\Sigma)/(x\sim \phi(x))\in \Alexnk$.


\noindent Case 1. Assume $\partial \Sigma=\varnothing$. Take two copies of
$\bar C_{\kappa}^R(\Sigma)$, marked as $\bar C_{\kappa}^R(\Sigma)_1$ and $\bar
C_{\kappa}^R(\Sigma)_2$, whose vertices are $p_1$ and $p_2$
respectively. Gluing along their boundaries by $\phi$, we obtain a double
space $\widehat{X}=\bar C_{\kappa}^R(\Sigma)_1\cup_\phi\bar C_{\kappa}^R(\Sigma)_2$.
By the gluing theorem ([Petrunin 97]), $\widehat{X}\in \text{Alex}^n(\kappa)$.

Now we extend the isometric $\mathbb Z_2$-action by $\phi$ on $\Sigma$
to an isometric $\hat{\mathbb Z_2}$-action on $\widehat X$ such that
$X=\widehat X/\hat{\mathbb Z_2}$, and thus $X\in\Alexnk$. For any $u\in
\bar C_\kappa^R(\Sigma)_1$, extend the geodesic $\geod{p_1u}_{\bar C_{\kappa}^R
(\Sigma)_1}$ to $u_1\in (\Sigma\times\{R\})_1$. Let $\hat \phi(u)$ be
the point on the geodesic $\geod{p_2\phi(u_1)}_{\bar C_{\kappa}^R
(\Sigma)_2}$ such that $|p_2\hat \phi(u)|=|p_1u|$ (so $\hat \phi:
\bar C^R_\kappa(\Sigma)_1\to \bar C^R_\kappa(\Sigma)_2$).
Switching the role of $\bar C_\kappa^R(\Sigma)_1$ and
$\bar C_\kappa^R(\Sigma)_2$, we extend $\phi$ to an isometric
involution $\hat \phi: \bar C^R_\kappa(\Sigma)_2\to \bar C^R_\kappa(\Sigma)_1$.
Clearly, $\hat \phi: \hat X\to \widehat X$ is an isometric involution such that
$X=\widehat {X}/\hat \phi$ .


\noindent Case 2. Assume $\partial \Sigma\neq\varnothing$. Let $\hat
\Sigma=\Sigma^+\cup \Sigma^-$ denote the double of $\Sigma$. We first
extend the isometric involution $\phi$ on $\Sigma$ to $\hat \phi:
\hat \Sigma\to \hat \Sigma$ by $\hat \phi(x_{\pm})=\phi(x)_{\mp}$, where
$x_+=x_-\in \Sigma$. We then define another isometric involution $\psi: \hat
\Sigma\to \hat \Sigma$ by the reflation on $\partial \Sigma$,
$\psi(x_{\pm})=x_{\mp}$. Then $\hat \psi(\hat \phi(x_{\pm}))=\hat\psi(\phi(x)_{\pm})
=\phi(x)_{\mp}=\hat \phi(x_{\mp})=\hat \phi(\hat \psi(x_{\pm}))$. This implies that
  $\hat \Sigma$ admits an $\Bbb Z_2\oplus \Bbb Z_2$-action. Clearly, the $\Bbb Z_2\oplus
  \Bbb Z_2$-action extends uniquely to an isometric $\Bbb Z_2\oplus \Bbb Z_2$-action on
$\bar C^r_\kappa(\hat \Sigma)$. By Case 1, we extends only the $\hat \phi$-action
to $\widehat X$ such that $\bar C^r_\kappa(\hat \Sigma)/x\sim \hat \phi(x)\in
\text{Alex}^n(\kappa)$. Then
  $X=[\bar C^r_\kappa(\hat \Sigma)/x\sim \hat \phi(x)]/\hat \psi\in \text{Alex}^n(\kappa)$.
\end{proof}

By Theorem B, the isometric classification of $X\in \mathcal A^r_\kappa(\Sigma)$
with relatively maximum volume reduces to the isometric classification of
all $(n-1)$-dimensional Alexandrov spaces $\Sigma$ with curv $\ge 1$
and the equivariant isometric $\Bbb Z_2$-actions on $\Sigma$. For $n=2$,
one easily gets a complete list:


\begin{cor} Any $2$-dimensional compact Alexandrov space with
curv $\ge \kappa$ and relatively maximum volume is isometric
to one of the following:
$$\bar C^r_\kappa(S^1_\theta)/\phi_i \hskip3mm (i=1,2,3),
\qquad \bar C^r_\kappa([0,\theta])/\psi_i \hskip3mm(i=1,2).$$
where $S^1_\theta$ denotes a circle of length $2\theta$
with $0<\theta\le \pi$, $\phi_i: S^1_\theta\to S^1_\theta$
(resp. $\psi_i: [0,\theta]\to [0,\theta]$) is trivial,
reflection or ancipital respectively for $i=1, 2$ and $3$
(resp. $i=1$ and $2$).
\end{cor}

\begin{example}[One-to-one Self-Gluing]\label{eg1} This is an example for self-gluing (c.f. [GP]). Let $Z=\mathbb D^2$ be a 2-dimensional flat unit disk. Then $\partial Z=\mathbb S^1(1)$ is a unit circle. Let $\phi:\partial Z\to\partial Z$ be a one-to-one map and $X=\mathbb D^2/x\sim\phi(x)$ be the glued space via identification $z\sim\phi(z)$. By Theorem B, $X$ is an Alexandrov space if and only if $\phi$ is a reflection, antipodal map or identity, where $X$ is homeomorphic to $\mathbb S^2$, $\mathbb RP^2$ and $\mathbb D^2$ respectively.
\end{example}

\begin{example}[Three Points Glued in a Self-Gluing]\label{eg2} Let $Z$ be a triangle. We identify points on each side via a reflection about the mid point, i.e., i.e., $\geod{Ab}$ glued with $\geod{Cb}$, $\geod{Ac}$ glued with $\geod{Bc}$, $\geod{Ba}$ glued with $\geod{Ca}$ and $A$, $B$ and $C$ are glued to one point and $A$, $B$ and $C$ are glued to one point. The glued space $X$ is a tetrahedron, which belongs to $\Alex^2(0)$.

  
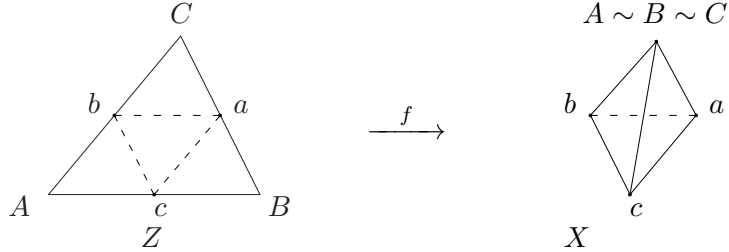
\begin{center}
  \bigskip
  \begin{picture}(300,90)
    \put(55,0){$Z$}
    \put(5,12){$A$}
    \put(66,85){$C$}
    \put(103,12){$B$}
    \put(90,51){$a$}
    \put(35,51){$b$}
    \put(60,12){$c$}
    \drawline(20,20)(100,20) 
    \drawline(20,20)(70,80) 
    \drawline(100,20)(70,80) 
    \put(60,20){\circle*{1}} 
    \put(85,50){\circle*{1}} 
    \put(45,50){\circle*{1}} 
    \dashline{3}(60,20)(85,50) 
    \dashline{3}(85,50)(45,50) 
    \dashline{3}(45,50)(60,20) 

    \put(135,45){$\begin{CD}@>f>> \end{CD}$}

    \put(215,0){$X$}
    \put(222,85){$A\sim B\sim C$}
    \put(270,51){$a$}
    \put(215,51){$b$}
    \put(240,12){$c$}
    \put(240,20){\circle*{1}} 
    \put(265,50){\circle*{1}} 
    \put(225,50){\circle*{1}} 
    \put(250,78){\circle*{1}} 
    \drawline(240,20)(265,50) 
    \dashline{3}(265,50)(225,50) 
    \drawline(225,50)(240,20) 
    \drawline(240,20)(250,78) 
    \drawline(265,50)(250,78) 
    \drawline(225,50)(250,78) 
  \end{picture}
  \end{center}
\end{example}

\section{Relatively almost maximum volume}

In the proof of Theorem C, we need the following result.

\begin{thm}[{Theorem 5.5 in [Br]}]\label{br5.5}
Let $M$ be a $G$-manifold, $G$ is a finite group. Assume that for a given
prime $p$ and all $p$-subgroups $P\subseteq G$ satisfies that
$$H_i(M^P;\Bbb Z_p)=0,\qquad i\le q \text{ (including $P=\{e\})$}.$$
Then $H_i(M/G;\Bbb Z_p)=0$ for
all $i\le q$. Moreover, if this holds for all prime $p$ and $H_i(M;\Bbb Z)=0$ for
$i\le q$, then $H_i(M/G;\Bbb Z)=0$ for $i\le q$.
\end{thm}


\begin{proof}[{\bf Proof of Theorem C}]

We first show that if $X\in \mathcal A^r_\kappa(\Sigma)$ with $\text{vol}(X)=
v(\Sigma, \kappa,r)$, then $X$ is homeomorphic to $S^n$ or $\Bbb CP^n$.

By Theorem B, $X$ is isometric to $\bar C^R_\kappa(\Sigma))/x\sim \phi(x)$,
$\phi: \Sigma\to \Sigma$ is an isometric involution. To determine the
homeomorphism type of $X$, we consider the double space $\widehat X=\bar
C^R_\kappa(\Sigma))^+\cup_\phi \bar C^R_\kappa(\Sigma))^-$. As seen in
the proof of Theorem B, $\hat X\in \text{Alex}^n(\kappa)$
and $\phi$ extends an isometric $\Bbb Z_2$-action on $\hat X$ such
that $\hat X/\Bbb Z_2$.

We claim that $\hat X$ is a homeomorphism sphere. First, $\hat X$ is a
topological manifold if every point $\hat q\in \partial
\bar C^R_\kappa(\Sigma))\hookrightarrow \hat X$ is a manifold point.
According to [Wu], a point $x$ in an Alexandrov space is a manifold point if
and only if $\Sigma_x$ is simply connected. Because $\Sigma_{\hat q}$ is
a suspension of $\Sigma_{\hat q}(\Sigma)$, $\hat q$ is a manifold point.
By the Poincar\'e conjecture (in all dimensions), our claim reduces
to that $\hat X$ is an integral homotopy sphere. Because $\hat X$ is a suspension,
$\hat X$ is simply connected, and thus it suffices to show that
$\hat X$ is a homology sphere. Because $\bar C^R_\kappa(\Sigma)$ is
contractible, from Mayer-Vietoris
exact sequence of $(\bar C^R_\kappa(\Sigma))^+,\bar C^R_\kappa(\Sigma))^-)$
it is easy to see that $\hat X$ is an integral homology sphere.

If the $\Bbb Z_2$-action is free, then $X=\hat X/\Bbb Z_2$ is homeomorphic to
$\Bbb RP^n$. Otherwise, $X$ is a simply connected topological manifold
(the induced map, $\pi_1(\hat X)\to \pi_1(X)$ is an onto map).
Again, it suffices to show that $X$ is an integral homology sphere. By Smith
theorem, the $\Bbb Z_2$-fixed point set $\hat X^{\Bbb Z_2}$ is an
$\Bbb Z_2$-homology sphere. By now we can apply Theorem 4.1 to conclude
the claim.

We then prove Theorem C by contradiction; assuming a sequence $X_i\in
\mathcal A^r_\kappa(\Sigma)$ such that $\text{vol}(X_i)>\text{vol}(C^R_\kappa(\Sigma))
-\epsilon_i$ ($\epsilon_i=i^{-1}$), and none of $X_i$ is homeomorphic to
$S^n$ or $\Bbb RP^n$. Without loss of generality, we may assume that
$(X_i,p_i)\overset{d_{GH}}\longrightarrow (X,p)\in \text{Alex}^n(\kappa)$, where
$X_i=\bar B_r(p_i)$. By Perelman's stability theorem ([Ka2], [Pe]), $X_i$ is homeomorphic
to $X$ for $i$ large. In particular, $X$ is a topological manifold. We claim
that $X\in \mathcal A^r_\kappa(\Sigma_p)$ satisfies that $\text{vol}(X)=v
(\Sigma_p,\kappa,r)$. By the above, we then conclude that $X$ is homeomorphic
to $S^n$ or $\Bbb RP^n$, and thus $X_i$ is homeomorphic to $X$ for $i$ large,
a contradiction.

To see the claim,
$$\text{vol}(X)=\lim_{i\to \infty}\text{vol}(X_i)=\lim_{i\to \infty}
(\text{vol}(C^R_\kappa(\Sigma))-\epsilon_i)=\text{vol}(C^R_\kappa(\Sigma)).$$
On the other hand, we shall construct a distance non-increasing map, $\phi: \Sigma\to \Sigma_p$.
Consequently, $\text{vol}(\Sigma_p)\le \text{vol}(\Sigma)$ and thus $$\text{vol}(X)\le
\text{vol}(C^R_\kappa(\Sigma_p))\le \text{vol}(C^R_\kappa(\Sigma))\le \text{vol}(X).$$
Let $A=\{v_i\}\subset \Sigma$ be a countable dense subset, and let $f_i: (X_i,p_i)\to (X,p)$
be a sequence of $\epsilon_i$-Gromov-Hausdorff approximation, $\epsilon_i\to 0$.
For $v_1$, the sequence $\{f_i(g\exp_{p_i}v)\}\subset X$ contains a converging
subsequence $f_{i_1}(g\exp_{p_{i_1}}q(v))\to x_1\in X$. Then $[px_1]=w_1\in \Sigma_p$
(which may not be unique). We define $\phi(v_1)=w_1$. For $v_2$ and $\{f_{i_1}\}$,
repeating the above we obtain $w_2\in \Sigma_p$ and define $\phi(v_2)=w_2$.
Iterating this process, we define a map $\phi: A\to \Sigma_p$, $\phi(v_i)=w_i$.
It is easy to check that $\phi$ is distance non-increasing and thus $\phi$ extends uniquely
to distance non-increasing map from $\Sigma$ to $\Sigma_p$.
\end{proof}

\section{Pointed Bishop-Gromov relative volume comparison}

Assuming the monotonicity in Theorem D, the rigidity part follows by Lemma \ref{bs.rig} and
Theorem \ref{open.iso}. For $p\in X\in\Alexnk$, let $A_R^r(p)$ (or briefly $A_R^r$)
denote the annulus $\{x\in X: r<|px|\leq R\}$,
$0\leq r<R$, and let $A_R^r(\Sigma_p)$ (or briefly $\tilde A_R^r$) denote the
corresponding annulus in $C_{\kappa}(\Sigma_p)$. Let $B_r$ denote $A_r^0$,
$\tilde B_r$ denote $\tilde A_r^0$. Let's recall the following two lemmas from
[LR].

\begin{lem}[{[LR]} Lemma 2.1]\label{lr2.1}

Let $\Sigma\in \text{Alex}^{n-1}(1)$ and $0<r\le \frac \pi{\sqrt \kappa}$. Then
$$\vol{C_\kappa^r(\Sigma)}
=\vol{\Sigma}\cdot \int^r_0\text{sn}^{n-1}_\kappa (t)\,dt.
$$

\end{lem}

\begin{lem}[{[LR]} Theorem B]\label{lrthmB}
Let $U$ be an open subset in $X\in\Alexnk$. Then there is a constant
$c(n)$ depending only on $n$ such that
$$V_{r_n}(\bar U)=V_{r_n}(U)=c(n)\cdot
\text{Haus}_n(U)=c(n)\cdot \text{Haus}_n(\bar U),$$
where $V_{r_n}$ and $\text{Haus}_n$ represent the $n$-dimensional
rough volume and Hausdorff measure respectively.
\end{lem}

\begin{lem}\label{bs.rig}
  If the monotonicity in Theorem B holds. then $$\frac{\text{vol\,}(B_r)}{\text{vol\,}(\tilde B_r)}
=\frac{\text{vol\,}(B_R)}{\text{vol\,}(\tilde B_R)}
$$
for some $0<r<R$ ($R\le\frac\pi{\sqrt\kappa}$ for $\kappa>0$) if
and only if $\text{vol\,}(B_R)=\text{vol\,}(\tilde B_R)$.
\end{lem}

\begin{proof}

Assume $\text{vol\,}(B_R)=\text{vol\,}(\tilde B_R)$. The desired equation follows by the monotonicity:

\begin{align*}
1&=\frac{\vol{B_R}}{\text{vol\,}(\tilde B_R)}
\le \frac{\vol{B_r}}{\text{vol\,}(\tilde B_r)}
\le \lim_{r\ge t\to 0}\frac{\vol{B_t}}{\text{vol\,}(\tilde B_t)}=1.
\end{align*}
Assume
$\dsp\frac{\text{vol\,}(B_r)}{\text{vol\,}(\tilde B_r)}
=\frac{\text{vol\,}(B_R)}{\text{vol\,}(\tilde B_R)},$ for some
$0<r<R$. Then for any $t<r$,
\begin{align*}
\frac{\text{vol\,}(B_{t})}{\text{vol\,}(A_R^r)}
+\frac{\text{vol\,}(A_R^t)}{\text{vol\,}(A_R^r)}
&=\frac{\text{vol\,}(B_{R})}
{\text{vol\,}(A_R^r)}
=\frac{\text{vol\,}(\tilde B_R)}
{\text{vol\,}(\tilde A_R^r)}
=\frac{\text{vol\,}(\tilde B_t)}
{\text{vol\,}(\tilde A_R^r)}
+\frac{\text{vol\,}(\tilde A_R^t)}
{\text{vol\,}(\tilde A_R^r)}.
\end{align*}
By the monotonicity, we have $\dsp\frac{\text{vol\,}(A_R^t)}{\text{vol\,}(A_R^r)}
\geq \frac{\text{vol\,}(\tilde A_R^t)}
{\text{vol\,}(\tilde A_R^r)}$. Also,
\begin{align*}
\frac{\text{vol\,}(B_{t})}{\text{vol\,}(A_R^r)}
&=\frac{\text{vol\,}(B_{t})}{\text{vol\,}(A_r^t)}
\cdot \frac{\text{vol\,}(A_r^t)}{\text{vol\,}(A_R^r)}
\geq \frac{\text{vol\,}(\tilde B_t)}
{\text{vol\,}(\tilde A_r^t)}
\cdot \frac{\text{vol\,}(\tilde A_r^t)}
{\text{vol\,}(\tilde A_R^r)}
= \frac{\text{vol\,}(\tilde B_t)}
{\text{vol\,}(\tilde A_R^r)}.
\end{align*}
Consequently
$\dsp\frac{\text{vol\,}(B_{t})}{\text{vol\,}(A_R^r)}
= \frac{\text{vol\,}(\tilde B_t)}
{\text{vol\,}(\tilde A_R^r)}$, or equivalently,
$\dsp\frac{\text{vol\,}(B_{t})}{\text{vol\,}(\tilde B_t)}
= \frac{\text{vol\,}(A_R^r)}
{\text{vol\,}(\tilde A_R^r)}.$
Let $t\rightarrow 0$, we get
$\text{vol\,}(A_R^r)=\text{vol\,}(\tilde A_R^r)$. Thus
$$1\ge\frac{\text{vol\,}(B_R)}{\text{vol\,}(\tilde B_R)}
\ge \frac{\text{vol\,}(A_R^r)}{\text{vol\,}(\tilde A_R^r)}
=1.$$
\end{proof}

By now, it remains to show the monotonicity in Theorem D.
We take an elementary approach by expressing the monotonicity as a form of ``Riemann
sum" (see (\ref{e4.1.1.4})) and using the Toponogov triangle comparison to
bound each term in terms of the desired form (see Corollary \ref{t4.1.7}).
To achieve this goal, we choose a special infinite partition (see (\ref{e4.1.1.4}) and (\ref{e4.1.8})).

We start the proof of Theorem D by deriving an equivalent form of
the monotonicity. For $0\le R_1<R_2<R_3$ ($< \frac \pi{\sqrt \kappa}$
when when $\kappa>0$), and $p\in X$, by Lemma \ref{lr2.1}, the monotonicity has the following integral form
$$\frac {\text{vol\,}(A_{R_3}^{R_1})}{\text{vol\,}(A_{R_2}^{R_1})}\le
\frac{\dsize \int^{R_3}_{R_1}\snk^{n-1}(t)\,dt}
{\dsize\int^{R_2}_{R_1}\snk^{n-1}(t)\,dt},$$
which is equivalent to
\begin{equation}
I_1=\log \left[\frac {\text{vol\,}(A_{R_3}^{R_1})}{\text{vol\,}(A_{R_2}^{R_1})}
\right]\le \log \left[\frac{\dsize \int^{R_3}_{R_1}\snk^{n-1}(t)\,dt}
{\dsize\int^{R_2}_{R_1}\snk^{n-1}(t)\,dt}\right]=I_2.
\label{e4.1.1.1}\end{equation}
Fixing a small $\delta>0$, let $m=[\frac {R_3-R_2}\delta]+1$, $\Delta=
\frac{R_3-R_2}m\approx \delta$, and $r_j=R_2+j\cdot \Delta$, $0\le j\le m$. Then
$$A_{R_2}^{R_1}=A_{r_0}^{R_1}\subset A_{r_1}^{R_1}\subset \cdots \subset A_{r_m}^{R_1}
=A_{R_3}^{R_1}.$$
Using the Taylor expansion $\log \dfrac 1x= 1-x+O((1-x)^2)$, we may
rewrite the left hand side of (\ref{e4.1.1.1}) as:
\begin{align}
I_1&=\sum_{j=1}^m\log \frac{\text{vol\,}(A_{R_1}^{r_j})}
{\text{vol\,}(A_{R_1}^{r_{j-1}})}
=\sum_{j=1}^m\left [\left(1-\frac{\text{vol\,}(A_{R_1}^{r_{j-1}})}
{\text{vol\,}(A_{R_1}^{r_j})}\right)+O(\delta^2)\right]
\label{e4.1.1.2}\\
&=\sum_{j=1}^m\frac{\text{vol\,}(A_{r_{j-1}}^{r_j})}
{\text{vol\,}(A_{R_1}^{r_j})}+O(\delta).
\notag\end{align}
Let $\phi(r)=\dsize
\int^r_{R_1} \snk^{n-1}(t)\,dt$. Then the right hand side of (\ref{e4.1.1.1})
can be written as:
\begin{align}
I_2&=\log \frac {\phi(R_3)}{\phi(R_2)}=\int^{R_3}_{R_2}\frac {\phi'(t)}
{\phi(t)}\,dt
\label{e4.1.1.3}\\
&=\sum_{j=1}^m\frac{\phi'(r_j)}{\phi(r_j)}\delta+\tau(\delta)
\notag\\&=
\sum_{j=1}^m\frac {\delta\cdot\snk^{n-1}(r_j)}{\int^{r_j}_{R_1}
\snk^{n-1}(t)\,dt}+\tau(\delta).
\notag\end{align}
Comparing (\ref{e4.1.1.1}) to (\ref{e4.1.1.2}) and (\ref{e4.1.1.3}), it's sufficient to show
\begin{equation}
\frac{\text{vol\,}(A_{r_{j-1}}^{r_j})}
{\text{vol\,}(A_{R_1}^{r_j})}\le \frac {\delta\cdot\snk^{n-1}(r_j)}{\int^{r_j}_{R_1}
\snk^{n-1}(t)\,dt}.
\label{e4.1.1.3+1}\end{equation}
We further divide $A_{R_1}^{r_j}$ into thinner annulus: given a monotonic
sequence $\{a_i\}_{i=1}^\infty\subset[0,1]$ such that $a_j\to 0$.
Then $\{a_ir_j\}_{i=1}^\infty$ is an infinite partition for $[0,r_j]$,
and (\ref{e4.1.1.3+1}) is equivalent to
\begin{equation}
  \frac{\text{vol\,}(A_{R_1}^{r_j})}{\text{vol\,}(A_{r_{j-1}}^{r_j})} =\sum_{i=1}^\infty
\frac{\text{vol\,}(A^{a_{i+1}r_j}_{a_ir_j})}{\text{vol\,}(A_{r_{j-1}}^{r_j})}
\ge \frac {\int^{r_j}_{R_1}
\snk^{n-1}(t)\,dt}{\delta\cdot\snk^{n-1}(r_j)}.
\label{e4.1.1.4}\end{equation}

To show (\ref{e4.1.1.4}), we need to estimate $\frac{\text{vol\,}(A^{a_{i+1}r_j}_{a_ir_j})}
{\text{vol\,}(A_{r_{j-1}}^{r_j})}$ from below (see Corollary \ref{t4.1.7}).
Assume $\delta$ is so small that $R-\delta>0$ and $r-\lambda \delta>0$. Let $x\in A_{R-\delta}^R$. We define
a map, $\phi: A_{R-\delta}^R\to A_{r-\lambda \delta}^r$, where $f(x)$ is the point on a
minimal geodesic $\geod{px}$ (if not unique, we pick one of them) such that
$$|pf(x)|=r-\lambda(R-|px|).$$
Because a geodesic in $X$ does not branch, $\phi$ is well-defined and is injective.

In the proof of Theorem D, the following is a main technical lemma, which
asserts that $\phi$ behaves like a bi-Lipschitz function.

\begin{lem}\label{t4.1.5}

Let $\delta>0$ sufficiently small, $\lambda=\frac{\snk r}{\snk R}$
and $\phi: \tilde A_{R-\delta}^R\to \tilde A_{r-\lambda \delta}^r$ be defined as the above. Then
$$c(\kappa,\delta)\cdot \lambda \le \frac {\snk \frac {|\phi(x)\phi(y)|}2}
{\snk\frac{|xy|}2}\le c(\kappa,\delta)^{-1}\lambda,$$
where $c(\kappa,\delta)=\begin{cases} 1 & \kappa=0,\\ 1-\frac{2\delta}
{\snk R+\delta} & \kappa>0,\\ 1-\delta\cdot
\frac{\cosh _\kappa R}R & \kappa<0\end{cases}$.
\end{lem}

Because the proof of Lemma \ref{t4.1.5} is technical and somewhat tedious, we will delay it to the end of this section.

\begin{lem}\label{t4.1.6}

Let $U$ and $V$ be two open subsets of $X\in \text{Alex}^n(\kappa)$, and let $\phi:
V\to U$ be an injection. If $\phi$ satisfies that $\snk\frac
{|\phi(x)\phi(y)|}2\ge c\cdot \snk \frac{|xy|}2$ for any $x,y\in V$,
then $\text{vol\,}(U)\ge c^n\cdot \text{vol\,}(V)$,
where $c$ is a constant.
\end{lem}

\begin{proof} By Lemma \ref{lrthmB}, it suffices to prove for rough volume.
Recall that the $n$-dimensional rough volume of a subset $V$ is
$$V_{r_n}(V)=\lim_{\epsilon\to 0}\epsilon^n\cdot \beta_V(\epsilon),$$
where $\beta_V(\epsilon)$ denotes the number of points in an $\epsilon$-net $\{x_i\}$
on $V$.

By the assumption, $\{\phi(x_i)\}$
is a $2\snk^{-1}
\left(c\cdot\snk\frac{\epsilon}{2}\right)$-net in $U$.
We get
$$\beta_U\left(2\snk^{-1}
\left(c\cdot\snk\frac{\epsilon}{2}\right)\right)\geq
\beta_V(\epsilon),$$
or as the following form:
$$\frac{\epsilon^n}{\left(2\snk^{-1}
\left(c\cdot\snk\frac{\epsilon}{2}\right)\right)^n}
\cdot
\left(2\snk^{-1}
\left(c\cdot\snk\frac{\epsilon}{2}\right)\right)^n
\cdot\beta_U\left(2\snk^{-1}
\left(\snk\frac{\epsilon}{2}\right)\right)
\geq
\epsilon^n\beta_V(\epsilon).
$$
Let $\epsilon\rightarrow 0$, we get $\frac{1}{c^n}V_{r_n}(U)\geq
V_{r_n}(V)$.
\end{proof}

\begin{cor}\label{t4.1.7}

Let $p\in X\in\Alexnk$, $\delta>0$ small. Then
$$\frac{\text{vol\,}(A_{r-\lambda \delta}^r)}{\text{vol\,}(A_{R-\delta}^R)}
\ge (1-\tau(\delta))\cdot \left(\frac{\snk r}{\text{sn}
_\kappa R}\right)^n.$$
\end{cor}

\begin{proof}
Consider the map $\phi: A_{R-\delta}^R\to A_{r-\lambda \delta}^r$ and
$\tilde\phi: \tilde A_{R-\delta}^R\to \tilde A_{r-\lambda \delta}^r$
defined as the above. For any $x,y\in A_{R-\delta}^R$, take two points
$\tilde x, \tilde y\in C_{\kappa}(\Sigma_p)$ such that $|\tilde o\tilde x|
=|px|$, $|\tilde o\tilde y|=|py|$ and $|\tilde x\tilde y|=|xy|$.
By condition B (see [BGP]), it's easy to see that $|f(x)f(y)|\ge|
\tilde\phi(\tilde x)\tilde\phi(\tilde y)|$. Thus by Lemma \ref{t4.1.5}, we have
$$\snk\frac
{|f(x)f(y)|}2\ge\snk\frac
{|\tilde\phi(\tilde x)\tilde\phi(\tilde y)|}2
\ge (1-\tau(\delta))\cdot \snk \frac{|\tilde x\tilde y|}2
=(1-\tau(\delta))\cdot \snk \frac{|xy|}2.$$
Then we get the desired estimate by Lemma \ref{t4.1.6}.
\end{proof}

\begin{proof}[{\bf Proof of the monotonicity in Theorem D}]\quad

Continuing from the early discussion, the proof reduces to verify (\ref{e4.1.1.4}).
We now take $\delta>0$ sufficiently small, and choose
the sequence $\{a_i\}_{i=0}^\infty$ as:
\begin{equation}
a_0=1, a_{i+1}=a_i-\frac{\snk (a_ir_j)}{r_j\cdot \snk r_j}
\cdot\delta,\qquad i=0, 1, \cdots
\label{e4.1.8}\end{equation}
Then
\begin{align*}
0<a_{i+1}
\le
\begin{cases}
(1-\frac \delta {r_j})a_i, & \text{ if }\kappa\ge 0,
\\
(1-\frac \delta{\snk r_j})a_i, & \text{ if }\kappa<0,
\end{cases}
\end{align*}
and thus $a_i\to 0$ and is monotonically decreasing.
For each $0\le i<\infty$ and $0\le j\le m$, consider the map,
$\phi: A_{r_j-\delta}^{r_j}\to A_{a_ir_j-\lambda_i\delta}^{a_ir_j} =A_{a_{i+1}r_j}^{a_ir_j}$,
with $\lambda_i=\frac{\snk(a_ir_j)}{\snk (r_j)}$.
By Corollary \ref{t4.1.7}, we obtain that
$$\frac{\text{vol\,}(A_{a_{i+1}r_j}^{a_ir_j}
)}{\text{vol\,}(A_{r_j-\delta}^{r_j})}\ge (1-\tau(\delta))
\left(\frac{\snk(a_ir_j)}{\snk r_j}\right)^n.$$
Observe that for $\delta\to 0$, $\{a_i\}$ will become more dense, and thus
we can take $N_\delta>0$ such that $a_{N_\delta}r_j\ge R_1$ and $a_{N_\delta}
r_j\to R_1$ as $\delta\to 0$. Summing up for $i=0,1,\cdots, N_\delta$, we get
\begin{align*}
\frac{\text{vol\,}(A_{r_j}^{R_1})}{\text{vol\,}(A_{r_j-\delta}^{r_j})}
&\ge\frac{\sum_{i=0}^{N_\delta} \text{vol\,}(A_{a_{i+1}r_j}^{a_ir_j}
)}{\text{vol\,}(A_{r_j-\delta}^{r_j})}
\\
&\ge\sum_{i=0}^{N_\delta} (1-\tau(\delta))
\left(\frac{\snk(a_ir_j)}{\snk r_j}\right)^n
\\
&\ge(1-\tau(\delta))\cdot \frac 1{\delta\cdot
\snk^{n-1}(r_j)}\sum_{i=0}^{N_\delta}\snk^{n-1}(a_ir_j)\cdot
\frac {\delta \cdot\snk(a_ir_j)}{\snk r_j}
\\
&=(1-\tau(\delta))
\cdot \frac 1{\delta\cdot\snk^{n-1}(r_j)}\left(\int^{r_j}_{R_1}
\snk^{n-1}(t)\,dt+\tau (\delta)\right)
\\
&=(1-\tau(\delta))\cdot
\frac {\dsize\int^{r_j}_{R_1}\snk^{n-1}(t)\,dt}{\delta\cdot
\snk^{n-1}(r_j)},
\end{align*}
or the following equivalent form:
$$\frac{\text{vol\,}(A_{r_j-\delta}^{r_j})}{\text{vol\,}(A_{r_j}^{R_1})}
\le (1+\tau(\delta))\cdot
\frac{\delta\cdot
\snk^{n-1}(r_j)}{\dsize\int^{r_j}_{R_1}\snk^{n-1}(t)\,dt}.
$$
Summing up for all $j$ and together with (\ref{e4.1.1.2}) and (\ref{e4.1.1.3}),
we get $$I_1+O(\delta)\le(1+\tau(\delta))I_2+\tau(\delta).$$ Let $\delta\to 0$,
we get the desired inequality.
\end{proof}

The rest of this section is devoted to a proof of Lemma \ref{t4.1.5}.
The following are some properties used in the proof.

\begin{lem}\label{t4.1.9}\quad

\begin{enumerate}
  \item For $\lambda\in [0,1]$ and $x\in [0,\pi]$, $\sin\lambda x\geq \lambda\sin x$.
\item For $\lambda\in [0,1]$ and $x\geq 0$, $\sinh\lambda x\leq \lambda\sinh x$.
  \item For $\lambda \geq 0$ and $x\geq 0$, $\frac{\sin \lambda x}{\lambda\sin x}\geq 1-(\lambda x)^2/6$.
  \item For $\lambda \geq 0$ and $x\geq 0$, $\frac{\sinh \lambda x}{\lambda\sinh x}\geq \frac{x}{\sinh x}\geq 1-x$.
  \item Let $\triangle pab$ be a triangle in $S^2_\kappa$. The cosine law can be written as
$$sn_\kappa^2\frac{|ab|}{2}
        =sn_\kappa^2\frac{|pa|-|pb|}{2}+\sin^2\frac{\measuredangle apb}{2}sn_\kappa|pa|sn_\kappa|pb|.
$$
\end{enumerate}
\end{lem}

\begin{proof}
(1) Let $h(x)=\sin\lambda x-\lambda\sin x$, then
$$h^\prime(x)=\lambda\cos\lambda x-\lambda\cos x
=\lambda(\cos\lambda x-\cos x)\geq 0
$$ since $0\leq\lambda x\leq x\leq \pi$.

(2) Let $h(x)=\sinh\lambda x-\lambda\sinh x$, then
$$h^\prime(x)=\lambda\cosh\lambda x-\lambda\cosh x
=\lambda(\cosh\lambda x-\cosh x)\leq 0
$$ since $0\leq\lambda x\leq x$.

(3) For $x>0$, one can show that $x\geq \sin x\geq
x-x^3/6$. Then
$$\frac{\sin \lambda x}{\lambda\sin x}
\geq \frac{\lambda x-(\lambda x)^3/6}{\lambda x}
=1-(\lambda x)^2/6.
$$

(4) The first equality is easy to see through
$\sinh\lambda x\geq \lambda x$. Obviously. the second equality is
true for $x\geq 1$. For $0<x<1$,
$$\sinh x=x+\frac{x^3}{6}+\cdots\leq x(1+x+x^2+\cdots)=\frac{x}{1-x}.$$

(5) Follows by trigonometric metric identities.
\end{proof}

\begin{proof}[{\bf Proof of Lemma \ref{t4.1.5}}]

By scaling, we only need to check for $\kappa=1,-1$ and $\kappa=0$.

Case 1. $\kappa=1$. Noting that
$$\frac{|px^\prime|-|py^\prime|}{|px|-|py|}
=\frac{\lambda(|px|-|py|)}{|px|-|py|}
=\lambda,
$$
by Lemma \ref{t4.1.9}(3) and $0\leq ||px|-|py||\leq\delta
<\frac{1}{2}\sin R$, we have
\begin{align*}
\sin\left(\frac{||px^\prime|-|py^\prime||}{2}\right)
&=
\sin\left(\lambda\cdot\frac{||px|-|py||}{2}\right)
\notag\\
&\geq
\left(1-\frac{(\lambda\delta)^2}{6}\right)\lambda\cdot
\sin\left(\frac{||px|-|py||}{2}\right)
\notag\\&\geq
\left(1-\frac{\delta^2}{6\sin^2 R}\right)\lambda\cdot
\sin\left(\frac{||px|-|py||}{2}\right)
\notag\\
&\geq \left(1-\frac{2\delta}{\sin R+\delta}\right)\lambda\cdot
\sin\left(\frac{||px|-|py||}{2}\right)
\notag\\
&=\tau_1\lambda\cdot
\sin\left(\frac{||px|-|py||}{2}\right).
\end{align*}
Thus
\begin{align}
\tau_1\lambda
\leq
\frac{\sin\left(\frac{||px^\prime|-|py^\prime||}{2}\right)}
{\sin\left(\frac{||px|-|py||}{2}\right)}
\leq
\frac{\lambda\frac{||px|-|py||}{2}}
{\sin\left(\frac{||px|-|py||}{2}\right)}
\leq
\lambda\cdot\frac{\delta}{\sin\delta}
\leq \tau_1^{-1}\lambda.
\label{e4.1.3.1}\end{align}
For any $x\in \tilde A_{R-\delta}^R$,
by Lemma \ref{t4.1.9}(1), we have
$$\sin |px^\prime|\geq\frac{|px^\prime|}{r}\sin r
\geq
\frac{r-\lambda\delta}{r}\sin r
=
\frac{r-\frac{\sin r}{\sin R}\delta}{r}\sin r
\geq
\left(1-\frac{\delta}{\sin R}\right)\sin r,
$$
Together with $\sin|px'|-\sin
r=2\sin\frac{|px'|-r}{2}\cos\frac{|px'|+r}{2}\leq r-|px'|\leq
\lambda\delta$, we get
$$\left(1-\frac{\delta}{\sin R}\right)\sin r
\leq
\sin|px'|\leq\sin
r+\lambda\delta=\left(1+\frac{\delta}{\sin R}\right)\sin r.$$
Similarly,
$$\sin|px|\geq \frac{|px|}{R}\sin R\geq
\frac{R-\delta}{R}\sin R\geq \left(1-\frac{\delta}{\sin
R}\right)\sin R$$
and
$$\sin|px|-\sin
R=2\sin\frac{|px|-R}{2}\cos\frac{|px|+R}{2}\leq R-|px|\leq \delta,$$
hence
$$\left(1-\frac{\delta}{\sin R}\right)\sin
R\leq\sin|px|\leq\sin R+\delta=\left(1+\frac{\delta}{\sin
R}\right)\sin R.$$
So
\begin{equation}
c_1\,\frac{\sin r}{\sin R}
\leq \frac{\sin|px^\prime|}{\sin|px|}
\leq c_1^{-1}\frac{\sin r}{\sin R}.
\label{e4.1.3.2}
\end{equation} Let $\theta=\measuredangle xpy$. Since
$\frac{|xy|}{2}\leq\frac{\pi}{2}$, by the cosine law and
inequalities (\ref{e4.1.3.1}), (\ref{e4.1.3.2}),
\begin{align*}
c_1^2\lambda^2
\leq\frac{\sin^2\frac{|x^\prime y^\prime|}{2}}
{\sin^2\frac{|x y|}{2}}
&=\frac{\sin^2\frac{|px^\prime|-|py^\prime|}{2}
+\sin^2\frac{\theta}{2}\sin|px^\prime|\sin|py^\prime|}
{\sin^2\frac{|px|-|py|}{2}
+\sin^2\frac{\theta}{2}\sin|px|\sin|py|}
\leq c_1^{-2}\lambda^2.
\end{align*}

Case 2, $\kappa=-1$. By Lemma \ref{t4.1.9}(2), $\lambda\delta
=\frac{\sinh r}{\sinh R}\cdot\frac{R}{\cosh R} < \frac{r}{R}\cdot
R=r$. Together with Lemma \ref{t4.1.9}(4), we get
\begin{equation}
\lambda\geq
\frac{\sinh\left(\frac{||px^\prime|-|py^\prime||}{2}\right)}
{\sinh\left(\frac{||px|-|py||}{2}\right)}
=
\frac{\sinh\left(\lambda\cdot\frac{||px|-|py||}{2}\right)}
{\sinh\left(\frac{||px|-|py||}{2}\right)}
\geq
\left(1-\delta\right)\lambda
\geq
c_{-1}\lambda,
\label{e4.1.3.3}\end{equation} since $\frac{\cosh R}{R}\geq
\frac{1+R^2/2}{R}>1$. If $\delta<\frac{R}{\cosh R} <R$, then
$\frac{\lambda\delta}{2r}<\frac{r}{R}\cdot\frac{\delta}{2r}
=\frac{\delta}{2R}<1$. Hence we can apply Lemma \ref{t4.1.9}(2) with
$\lambda=\frac{\sinh r}{\sinh R}\leq \frac{r}{R}$, to get
\begin{align*}
\frac{\sinh r-\sinh(r-\lambda\delta)}{\sinh r}
\leq \frac{2\sinh (\lambda\delta/2)\cdot\cosh r}{\sinh r}
\leq \frac{\lambda\delta}{r}\cdot\cosh r
\leq \frac{\delta\cdot\cosh R}{R},
\end{align*} thus
$$\sinh(r-\lambda\delta)\geq
\left(1-\delta\cdot\frac{\cosh R}{R}\right)\sinh r.$$ For $x'\in
\tilde A_{r-\lambda\delta}^r$, $\left(1-\delta\cdot\frac{\cosh
R}{R}\right)\sinh r
\leq\sinh(r-\lambda\delta)
\leq\sinh|px^\prime|
\leq\sinh r$.
For $x\in \tilde A_{R-\lambda\delta}^R$,
$$\frac{\sinh R-\sinh(R-\delta)}{\sinh R}
\leq \frac{2\sinh(\delta/2)\cosh R}{\sinh R}
\leq \frac{\delta\cdot\cosh R}{R},
$$
and $\left(1-\delta\cdot\frac{\cosh R}{R}\right)\sinh R
\leq\sinh(R-\lambda\delta)
\leq\sinh|px|
\leq\sinh R$.
Then
\begin{equation}
c_{-1}\frac{\sinh r}{\sinh R}
\leq
\frac{\sinh|px^\prime|}{\sinh|px|}
\leq
c_{-1}^{-1}\frac{\sinh r}{\sinh R}
\label{e4.1.3.4}.
\end{equation}
By inequalities
(\ref{e4.1.3.3}), (\ref{e4.1.3.4}) and the cosine law, we get
\begin{align*}
c_{-1}^2\lambda^2
\leq\frac{\sinh^2\frac{|x^\prime y^\prime|}{2}}
{\sinh^2\frac{|x y|}{2}}
&=\frac{\sinh^2\frac{|px^\prime|-|py^\prime|}{2}
+\sin^2\frac{\theta}{2}\sinh|px^\prime|\sinh|py^\prime|}
{\sinh^2\frac{|px|-|py|}{2}
+\sin^2\frac{\theta}{2}\sinh|px|\sinh|py|}
\leq c_{-1}^{-2}\lambda^2.
\end{align*}

Case 3. $\kappa=0$. This is straight forward.
\end{proof}

%

\vskip 30mm

\bibliographystyle{amsalpha}


\end{document}